\documentclass{amsart}
\usepackage{enumitem}
\usepackage{amsmath, amsfonts, amssymb, graphicx, tikz, tikz-cd, bbold, amsthm, calc, textcomp, xfrac, wasysym}
\usepackage[mathscr]{euscript}
\usetikzlibrary{matrix,arrows,decorations.pathmorphing}
\theoremstyle{plain} 
\newtheorem{theorem}[equation]{Theorem} 
\newtheorem{lemma}[equation]{Lemma}
\newtheorem{proposition}[equation]{Proposition}
\newtheorem{corollary}[equation]{Corollary}
\newtheorem*{proposition*}{Proposition}
 \newtheorem*{theorem*}{Theorem}
 \newtheorem*{corollary*}{Corollary}
\theoremstyle{remark}
 \newtheorem{remark}[equation]{Remark}
 \newtheorem{example}[equation]{Example}
 \newtheorem{notation}[equation]{Notation}
\theoremstyle{definition}
 \newtheorem{definition}[equation]{Definition}
 \newtheorem*{definition*}{Definition}
 \newtheorem*{acknowledgment}{Acknowledgment}

\newcommand{\N}{\mathbb{N}}
\newcommand{\Z}{\mathbb{Z}}
\newcommand{\K}{\mathscr{K}}
\newcommand{\cat}[1]{\mathscr{#1}}
\newcommand{\p}{\mathscr{P}}
\newcommand{\T}{\otimes}
\newcommand{\1}{\mathbb{1}}
\newcommand{\D}{\text{D}^{\text{perf}}}
\newcommand{\kk}{\mathbb{k}}
\newcommand{\V}{\mathcal{V}}

\DeclareMathOperator\Spc{Spc}
\DeclareMathOperator\Spec{Spec}
\DeclareMathOperator\Proj{Proj}
\DeclareMathOperator\pr{pr}
\DeclareMathOperator\inc{incl}

\DeclareMathOperator\im{im}

\newcommand{\Mod}[2]{#2\!-\!\operatorname{Mod}_{#1}}

\newcommand{\supp}{\operatorname{supp}}
\newcommand{\Res}[2]{\operatorname{Res}_{#2}^{#1}}
\newcommand{\M}[1]{\Mod {\K}{#1}}
\renewcommand{\mod}[1]{\kk#1\!-\!\operatorname{mod}}
\newcommand{\stab}[1]{\kk#1\!-\!\operatorname{stab}}

\newcommand{\Db}[1]{\operatorname{D}^b(\mod {#1})}

\newcommand{\norm}[1]{\operatorname{norm}^G_{#1}}

\newcommand{\divides}{\bigm|}
\newcommand{\ndivides}{%
  \mathrel{\mkern.5mu 
    \ooalign{\hidewidth$\big|$\hidewidth\cr$\nmid$\cr}%
  }%
}
\numberwithin{equation}{section}
\usepackage{hyperref}

\begin{document}

\title{Quasi-Galois theory in symmetric-monoidal categories}
\author{Bregje Pauwels}

\begin{abstract}
	Given a ring object $A$ in a symmetric monoidal category, we investigate what it means for the extension $\1\rightarrow A$ to be (quasi-)Galois. In particular, we define splitting ring extensions and examine how they occur. 
	Specializing to tensor-triangulated categories, we study how extension-of-scalars along a quasi-Galois ring object affects the Balmer spectrum. We define what it means for a separable ring to have constant degree, which is a necessary and sufficient condition for the existence of a quasi-Galois closure. 
	Finally, we illustrate the above for separable rings occurring in modular representation theory. 
\end{abstract}

\maketitle
\tableofcontents

\section*{Introduction}
Classical Galois theory is the study of field extensions $l/k$ through the group of automorphisms of $l$ that fix~$k$. For an irreducible and separable polynomial $f\in k[x]$, the \emph{splitting field} of $f$ over $k$ is the smallest extension over which $f$ decomposes into linear factors. 
In other words, the splitting field $l$ of $f$ is the smallest extension of $k$ such that~$l\T_k k[x]/(f)\cong l^{\times\deg(f)}$.
The field extension $l/k$ is often called \emph{quasi-Galois}\footnote{see Bourbaki~\cite[\S 9]{bourbaki}. In the literature, a quasi-Galois extension is sometimes called normal or Galois, probably because these notions coincide when $l/k$ is separable and finite.} if $l$ is the splitting field for some polynomial in~$k[x]$.

In this paper, we adapt the above ideas to the context of ring objects in a symmetric monoidal category $(\K,\T,\1)$, with special emphasis on tensor-triangulated categories.
That is, our analogue of a field extension will be a monoid $\eta : \1 \rightarrow A$ in~$\K$ with associative commutative multiplication $\mu:A\T A\rightarrow A$. We call $A$ a \emph{ring in~$\K$}, and moreover assume that $A$ is \emph{separable}, which means $\mu$ has an $A,A$-bilinear right inverse $A\rightarrow A\T A$.

Separable ring objects play an important(though at times invisible) role in various areas of mathematics.
In algebraic geometry, for instance, they appear as \'etale extensions of quasi-compact and quasi-separated schemes, see~\cite[Th.3.5]{bneemanthomason}. More precisely, given a separated \'etale morphism $f:V\rightarrow X$, the object $A:=Rf_{*}(\mathcal{O}_V)$ in $D^{\text {qcoh}}(X)$ is a separable ring, and we can understand $D^{\text {qcoh}}(V)$ as the category of $A$-modules in~$D^{\text {qcoh}}(X)$.
In representation theory, we can let $\K(G)$ be the (derived or stable) module category of a group $G$ over a field~$\kk$, and consider a subgroup $H<G$. In ~\cite{bstacks}, Balmer showed there is a separable ring $A^G_H$ in $\K(G)$  such that the category of $A^G_H$-modules in $\K(G)$ coincides with $\K(H)$, and such that the restriction functor $\Res{G}{H}:\K(G)\rightarrow \K(H)$ is just extension-of-scalars along~$A^G_H$.
In the same vein, extension-of-scalars along a separable ring recovers restriction to a subgroup in equivariant stable homotopy theory, in equivariant $KK$-theory and in equivariant derived categories, see~\cite{bds}.
For more examples of separable rings in stable homotopy categories, we refer to~\cite{bakerrichter}~and~\cite{rognes}.

Thus motivated, we study how much Galois theory carries over. 
The generalisation of Galois theory from fields to rings originated with Auslander and Goldman in~\cite[App.]{AG}. 
For more generalizations in various directions, see~\cite{chasesweedler, hess, noncommutative}. In particular, Rognes~\cite{rognes} introduced a Galois theory up-to-homotopy.

Recall that we call a ring $A$ in $\K$ \emph{indecomposable} if it doesn't decompose as a product of nonzero rings. 
Separable ring objects have a well-behaved notion of degree, see~\cite{bdegree}, and our first Galois-flavoured result~(Theorem~\ref{th:nrmorphisms}) shows that the number of ring endomorphisms of a separable indecomposable ring in $\K$ is bounded by its degree.
If $A$ is a ring in $\K$ and $\Gamma$ is a group of ring automorphisms of $A$, we call $A$ \emph{quasi-Galois in $\K$ with group $\Gamma$} if
the $A$-algebra homomorphism
$$\lambda_{\Gamma}: 
A\T A \longrightarrow \prod_{\gamma\in \Gamma} A$$
defined by $\pr_{\gamma} \lambda_{\Gamma}=\mu (1 \T \gamma)$ is an isomorphism.
An indecomposable ring $A$ is quasi-Galois in $\K$ for some group $\Gamma$ if and only if $A$ has exactly $\deg(A)$ ring endomorphisms in~$\K$ (see~Theorem~\ref{th:galoisequivalences}). In that case, $\Gamma$ contains all ring endomorphisms of $A$ in~$\K$.

\begin{definition*}
	Let $A$ and $B$ be rings of finite degree in~$\K$. We say $B$ \emph{splits} $A$ if $B\T A\cong B^{\times \deg(A)}$ as (left) $B$-algebras in $\K$.  
	We call an indecomposable ring $B$ a \emph{splitting ring} of~$A$ if~$B$ splits~$A$ and any ring morphism $C\rightarrow B$, where~$C$ is an indecomposable ring splitting~$A$, is an isomorphism. 
\end{definition*}

Under mild conditions on $\K$,~Corollary~\ref{cor:galoissplitring} shows $B$ is quasi-Galois in $\K$ if and only if $B$ is a splitting ring of some separable ring $A$ in $\K$; our terminology matches classical field theory.
Moreover, Proposition~\ref{prop:splittingring} shows that
every separable ring in~$\K$ has (possibly multiple) splitting rings.

If in addition, $\K$ is tensor-triangulated, we can say more about the way splitting rings arise. 
In~\cite{bspectrum}, Balmer introduced the \emph{spectrum} of a tensor-triangulated category $\K$, a topological space in which every $x\in\K$ has a \emph{support} $\supp(x)\subset\Spc(\K)$.
The Balmer spectrum provides an algebro-geometric approach to the study of triangulated categories, and a complete description of the spectrum is equivalent to a classification of the thick $\T$-ideals in the category.

For the remainder of the introduction, we assume $\K$ is tensor-triangulated
and nice (say, $\Spc(\K)$ is Noetherian or $\K$ satisfies Krull-Schmidt).
If $A$ is a separable ring in $\K$,
the \emph{Eilenberg-Moore category} $\M A$ of $A$-modules in $\K$ remains tensor-triangulated and extension of scalars is exact, see~\cite[Cor.4.3]{bseparability}.
We can thus extend scalars along a separable ring without leaving the tensor-triangulated world or descending to a model category.
If $A$ is quasi-Galois with group $\Gamma$ in~$\K$, then $\Gamma$ acts on $\M A$ and on the spectrum~$\Spc(\M A)$. Theorem~\ref{th:galoisspectrum} shows that $\supp(A)\subset \Spc(\K)$ is given by the $\Gamma$-orbits of $\Spc(\M A)$.
In particular, 
we recover $\Spc(\K)$ from $\Spc(\M A)$ if $\supp (A)=\Spc(\K)$, which happens exactly when $A\T f=0$ implies $f$ is $\T$-nilpotent for every morphism $f$ in $\K$.

Recall that for a quasi-Galois field extension $l/k$, any irreducible polynomial $f\in k[x]$ with a root in $l$ splits in $l$, see~\cite{bourbaki}. 
Proposition~\ref{prop:galoissplit} provides us with a tensor triangular analogue:

\begin{proposition*}
	Let $A$ be a separable ring in $\K$ such that the spectrum $\Spc(\M A)$ is connected, and suppose $B$ is an $A$-algebra with $\supp(A)=\supp(B)$. If $B$ is quasi-Galois in~$\K$, then $B$ splits~$A$.
\end{proposition*} 

Finally, Theorem~ \ref{th:splittingring} shows that certain rings in $\K$ have a quasi-Galois closure. Given $\p\in \Spc(\K)$, we consider \emph{the local category $\K_{\p} $ at $\p$}, that is the idempotent completion of the Verdier quotient $\K\diagup \p$. 
We say a ring $A$ has \emph{constant degree in~$\K$}  if the degree of $A$ as a ring in $\K_{\p}$ is the same for every prime $\p\in \supp(A)$.

\begin{theorem*} 
	If $A$ has constant degree in $\K$ and the spectrum $\Spc(\M A)$ is connected, then $A$ has a unique splitting ring $A^{*}$. Furthermore, $\supp(A)= \supp(A^{*})$ and  $A^{*}$ is the quasi-Galois closure of~$A$ in~$\K$. 
	That is, for any $A$-algebra $B$ that is quasi-Galois in~$\K$ with $\supp(A)=\supp(B)$, there exists a ring morphism $A^{*}\rightarrow $B.
\end{theorem*}

We conclude this paper by computing degrees and splitting rings for the separable rings $A^G_H:=\kk(G/H)$ mentioned before.
Here, $H<G$ are finite groups and $\kk$ is a field with characteristic $p$ dividing~$|G|$.
The degree of $A^G_H$ in $\Db G$ is simply $[G:H]$ and $A^G_H$ is quasi-Galois if and only if $H$ is normal in~$G$. Accordingly, the quasi-Galois closure of $A^G_H$ in $\Db G$ is the ring $A^G_N$, where $N$ is the normal core of $H$ in $G$ (see~Corollary~\ref{cor:Dbqg}).
On the other hand, Proposition~\ref{prop:stabqg}, shows the degree of $A^G_H$ in $\stab G$ is the greatest $0\leq n\leq [G:H]$ such that there exist distinct $[g_1], \ldots, [g_n]$ in $H\backslash G$ with $p$ dividing $|H^{g_1}\cap\ldots\cap H^{g_n}|$.   
In that case, the splitting rings of $A^G_H$ are 
exactly the 
$A^G_{H^{g_1}\cap\ldots\cap H^{g_{n}}}$ with $g_1,\ldots, g_n$ as above.


\begin{acknowledgment}
	I am very thankful to my advisor Paul Balmer for valuable ideas and helpful comments.
\end{acknowledgment}

\section{The Eilenberg-Moore Category}

\begin{definition}
	Let $\K$ be an additive category. We say $\K$ is \emph{idempotent-complete} if for all $x\in\K$, any morphism $e:x\rightarrow x$ with $e^2=e$  yields a decomposition $x\cong x_1\oplus x_2$ under which $e$ becomes~$\left( \begin{smallmatrix} 1&0 \\ 0&0 \end{smallmatrix} \right)$. 
	Every additive category $\K$ can be embedded in an idempotent-complete category $\K^{\natural}$ in such a way that $\K\hookrightarrow \K^{\natural}$ is fully faithful and every object in $\K^{\natural}$ is a direct summand of some object in~$\K$. We call $\K^{\natural}$ the \emph{idempotent-completion} of~$\K$, and~\cite{bidempotentcompletion} shows that $\K^{\natural}$ stays triangulated if $\K$ was.
\end{definition}

\begin{notation}
	Throughout, $( \K, \T, \1)$ denotes an 
	idempotent-complete symmetric monoidal category.
	For objects $x_1,\ldots,x_n$  in $\K$ and a permutation $\tau\in S_n$, we also write $\tau:x_1\T\ldots \T x_n\rightarrow x_{\tau(1)}\T\ldots\T x_{\tau(n)}$ to denote the isomorphism that permutes the tensor factors.
\end{notation}

\begin{definition}
	A \emph{ring object} $A\in \K$ is a monoid $(A,\mu:A\T A\rightarrow A, \eta:\1\rightarrow A)$ with associative multiplication $\mu$ and two-sided unit $\eta$. We call $A$ \emph{commutative} if $\mu (12)=\mu$. All ring objects in this paper will be commutative and we often simply call $A$ a \emph{ring in~$\K$}. For rings $A$ and $B$ in $\K$, a \emph{ring morphism} $f:A\rightarrow B$ is a morphism  in $\K$ that is compatible with the ring structure. 

	A \emph{(left) $A$-module} is a pair $(x\in\K,\varrho:A\T x\rightarrow x)$,  where the action $\varrho$ is compatible with the ring structure in the usual way. \emph{Right $A$-modules} and \emph{$A,A$-bimodules} are defined analoguesly.  

	The \emph{Eilenberg-Moore category} $\M A$ has left $A$-modules as objects and $A$-linear morphisms, which are defined in the usual way. 
	Every object $x\in\K$ gives rise to a \emph{free $A$-module} $F_A(x)=A \T x$ with action given by $\varrho : A \T A \T x\xrightarrow{\mu\T 1}A\T x$.
	We call the functor $F_A:\K\rightarrow \M A$ \emph{the extension-of-scalars}, and write $U_A$ for its forgetful right adjoint:
	$$\begin{tikzcd}[row sep=small]
     \K \arrow[xshift=-1.2ex, swap]{dd}{F_A} \\ \dashv  \\ \M A\arrow[xshift=1.2ex, swap]{uu}{U_A} .
	 \end{tikzcd}$$

	A ring $A$ in $\K$ is \emph{separable} if the multiplication map $\mu$ has an $A,A$-bilinear section~$\sigma:A\rightarrow A\T A$. That is,
	$\mu\sigma=1_A$ and the diagram
	$$\begin{tikzcd}[column sep=large]
	& A\T A \arrow[swap]{ld}{\sigma\T 1}
	\arrow{rd}{1\T\sigma}
	\arrow{d}{\mu}& \\
	A\T A\T A \arrow[swap]{rd}{1\T \mu}
	& A \arrow{d}{\sigma}
	& A\T A\T A
	\arrow{ld}{\mu\T 1}\\
	& A\T A &\end{tikzcd}$$
	commutes.
	
	If $A$ and $B$ are rings in~$\K$ and $h:A~\rightarrow~B$ is a ring morphism, we say that $B$ is an \emph{$A$-algebra}. 
	We can equip $B$ with the usual $A$-module structure via $h$, and write $\overline{B}$ for the corresponding object in $\M A$.
\end{definition}

\begin{remark}
	The module category $\M A$ is idempotent-complete whenever $\K$ is idempotent-complete. 
\end{remark}

\begin{example}\label{ex:etale}
	Let $R$ be a commutative ring and consider the category $R-\!\!\operatorname{mod}$ of finitely generated $R$-modules. Let $A$ be a commutative projective separable $R$-algebra. By~\cite[Prop.2.2.1]{demeyeringraham}, $A$ is finitely generated as an $R$-module, so $A$ defines a separable ring object in $R-\!\!\operatorname{mod}$.
	On the other hand, we can think of $A=A[0]$ as a separable ring object in $\D(R)$, the homotopy category of bounded complexes of finitely generated projective $R$-modules.  Note that the category of $A$-modules in~$\D(R)$ is equivalent to $\D(A)$ by~\cite[Th.6.5]{bseparability}. 
\end{example}

\begin{notation}
	Let $A$ and $B$ be rings in~$\K$. The ring structure on $A\T B$ is given by $(\mu_A\T\mu_B)(23):(A\T B)^{\T 2}\rightarrow(A\T B)$.  We write $A^e$ for the enveloping ring $A\T A^{\text{op}}$, so that left $A^e$-modules are just $A,A$-bimodules. 
	We write $A\times B$ for the ring $A\oplus B$ with component-wise multiplication.
\end{notation}

\begin{remark}
	If $A$ and $B$ are separable rings in~$\K$, then so are $A^e$, $A\T B$ and $A\times B$. Conversely, $A$ and $B$ are separable whenever $A\times B$ is separable.
\end{remark}

\begin{remark}\label{rem:endo}		
	Let $A$ be a ring in $\K$. Note that every (left) $A$-linear endomorphism $A\rightarrow A$ is in fact~$A^e$-linear, by commutativity of $A$. What is more, any two $A$-linear endomorphisms $A\rightarrow A$ commute.
	
\end{remark}

\begin{definition}\label{def:indecomposable}
	We call a nonzero ring $A$ in $\K$ \emph{indecomposable} if the only idempotent $A$-linear endomorphisms $A\rightarrow A$ in $\K$ are the identity $1_A$ and $0$. In other words, $A$ is indecomposable if it doesn't decompose as a direct sum of nonzero $A^{e}$-modules. By the following lemma, this is equivalent to saying $A$ doesn't decompose as a product of nonzero rings.
\end{definition}

\begin{lemma}\label{lem:AAtoring}\emph{(\cite[Lem.2.2]{bdegree}).}
	Let $A$ be a ring in~$\K$. Suppose there is an $A^e$-linear isomorphism $h:A\xrightarrow{\sim} B\oplus C$ for some $A^e$-modules
	$B, C$ in $\K$. Then $B$ and $C$ admit unique ring structures under which $h$ becomes a ring isomorphism $h:A\xrightarrow{\sim}~B~\times~C$.
\end{lemma}

Let $(A, \mu, \eta)$ be a separable ring in~$\K$ with separability morphism $\sigma$. In what follows, we define a tensor structure $\T_A$ on $\M A$ under which extension-of-scalars becomes monoidal. The following results all appear in~\cite[\S 1]{bdegree}. For detailed proofs, see~\cite[\S1.1]{bregjethesis}.
Let $(x, \varrho_1)$ and $(y, \varrho_2)$ be $A$-modules. Here, we can write $\varrho_2$ to indicate both a left and right action of $A$ on $y$, as $A$ is commutative. 
Seeing how the endomorphism
$$
v:\begin{tikzcd}
x\T y 
\arrow{r}{1\T \eta \T 1}
& x \T A \T y
\arrow{r}{1\T \sigma \T 1}
& x \T A \T A\T y 
\arrow{r}{\varrho_1\T \varrho_2} 
& x\T y
\end{tikzcd}$$
is idempotent and
$\K$ is idempotent-complete,
we can define $x\T_A y$ as the direct summand $\im(v)$ of $x\T y$.
We get a split coequaliser in~$\K$,
$$
\begin{tikzcd}[column sep=large]
x\T A \T y 
\arrow[yshift=0.5ex]{r}{\varrho_1 \T 1}
\arrow[yshift=-0.5ex]{r}[swap]{1 \T \varrho_2}
&x \T y
\arrow[two heads]{r}
&  x\T_A y,
\end{tikzcd} $$
and $A$ acts on $x\T_A y$ by
$$ \begin{tikzcd}
A\T x\T_A y 
\arrow[hook]{r}
& A\T x\T y
\arrow{r}{\varrho_1 \T 1 }
& x\T y
\arrow[two heads]{r}
& x\T_A y. 
\end{tikzcd}$$

\begin{proposition}
	The tensor product $\T_A$ yields a symmetric monoidal structure on $\M A$ under which $F_A$ becomes monoidal. We will write $\1_A=A$ for the unit object in $\M A$.
\end{proposition}

\begin{remark}\label{rem:algebracorrespondence}
	Let $A$ be a separable ring in $\K$. There is a one-to-one correspondence between $A$-algebras $B$ in $\K$ and rings $\overline{B}$ in $\M A$. Moreover, $B$ is separable if and only if $\overline{B}$ is.
\end{remark}

\begin{remark}\label{rem:F_h}
	Let $A$ be a separable ring in $\K$ and suppose $B$ is an $A$-algebra via $h:A~\rightarrow~B$. For every $A$-module $x$, we let $B$ act on the left factor of $F_h(x):=\overline{B}\T_A x $ as usual. This defines a functor
	$F_h: \M A \longrightarrow \M B$ and the following diagram commutes up to isomorphism:
	$$\begin{tikzcd}[column sep=small]
	&& \K 
	\arrow[swap]{lld}{F_A} 
	\arrow{rrd}{F_B} &&  \\
	\M A  
	\arrow{rrrr}{F_h}
	&&&& \M B.
	\end{tikzcd}$$
	Note also that $F_{gh}\cong F_g F_h$ for any ring morphism $g:B\rightarrow C$.
\end{remark}

\begin{proposition}\label{prop:diagramequivalence}
	Let $A$ be a separable ring in $\K$ and suppose $B$ is a separable $A$-algebra, say $\overline{B}\in \cat L:=\M A$. There is an equivalence $\M B \simeq \Mod {\cat L}{\overline{B}}$
	such that 
	$$
	\begin{tikzcd}
	\K
	\arrow{r}{F_A}
	\arrow{d}{F_B} 
	& \cat L
	\arrow{d}{F_{\overline{B}}}\\
	\M B
	\arrow{r}{\simeq}
	& \Mod {\cat L}{\overline{B}},
	\end{tikzcd}$$
	commutes up to isomorphism.
\end{proposition}

\section{Separable rings}

\begin{proposition}\label{prop:uniquedecomposition}
	Let $A$ be a separable ring in~$\K$.  If $A\cong B\times C$ for rings $B, C$ in $\K$, then any indecomposable ring factor of $A$ is a ring factor of $B$ or $C$.
	In particular, if $A$ can be written as a product of indecomposable $A$-algebras  $A\cong A_1 \times \ldots \times A_n$, this decomposition is unique up to isomorphism.
\end{proposition} 

\begin{proof}
	Suppose $A_1\in\K $ is an indecomposable ring factor of $A$, say $A\cong A_1 \times A_2$ for some ring $A_2$ in~$\K$. The category $\M A$ decomposes as
	$$\M A\cong \M{A_1} \times \M{A_2},$$ 
	with $\1_A$ corresponding to $(\1_{A_1}, \1_{A_2})$.
	Accordingly, the $A$-algebras $\overline{B}$ and $\overline{C}$ correspond to $(B_1, B_2)$ and $(C_1, C_2)$ respectively, with $B_i,C_i$ in $\M{A_i}$ for~$i=1,2$
	, such that $\overline{B}\cong B_1\times B_2$ and $\overline{C}\cong C_1\times C_2$ in~$\M A$.
	Given that $ \1_A\cong \overline{B}\times \overline{C}$, we see $\1_{A_1}\cong B_1 \times C_1$, hence $A_1\cong B_1$ or $A_1\cong C_1$. 
\end{proof}

\begin{lemma}\label{lem:idempotent}
	Let $A$ be a separable ring in~$\K$.
	\begin{enumerate}[label=(\alph*)]
	\item For every ring morphism $\alpha:A\rightarrow \1$, there exists a unique idempotent~$A$-linear morphism $e : A\rightarrow A$ such that $\alpha e=\alpha$ and $e\eta \alpha=e$.

	\item Suppose $\1$ is indecomposable. If $\alpha_i:A\rightarrow \1$ are distinct ring morphisms for $1\leq i\leq n$, with corresponding idempotent morphisms  $e_i: A\rightarrow A$ as above, then $e_i  e_j =\delta_{i,j} e_i$ and $\alpha_i e_j=\delta_{i,j} \alpha_i$.
\end{enumerate}
\end{lemma}

\begin{proof}
	Let $\sigma$ be a separability morphism for $A$. To show (a), consider the $A$-linear map $e:=(\alpha\T 1)\sigma:A\rightarrow A$. 
	Clearly, $\alpha e=\alpha (\alpha\T 1)\sigma=\alpha\mu\sigma=\alpha$, and the diagram
	$$\begin{tikzcd}[column sep=large]
	A
	\arrow{r}{\sigma}
	\arrow[-, double equal sign distance]{dd}
	&A\T A
	\arrow{r}{\alpha\T 1}
	\arrow{d}{1\T \sigma}
	& A 
	\arrow{d}{\sigma}\\
	&A\T A\T A
	\arrow{r}{\alpha\T 1\T 1}
	\arrow{d}{\mu\T 1}
	&A\T A
	\arrow{d}{\alpha\T 1}\\
	A
	\arrow{r}{\sigma}
	&A\T A 
	\arrow{r}{\alpha\T 1}
	& A
	\end{tikzcd}$$
	shows $e$ is idempotent.
	Since
	$$\begin{tikzcd}[column sep=large]
	A
	\arrow{r}{\alpha}
	\arrow{d}{1\T\eta}
	&\1
	\arrow{r}{\eta}
	&  A
	\arrow{d}{\sigma}
	\\
	A\T A
	\arrow{r}{1\T\sigma}
	\arrow{d}{\mu}
	& A \T A\T A
	\arrow{r}{\alpha\T 1\T  1}
	\arrow{d}{\mu\T 1}
	&A\T A
	\arrow{d}{\alpha\T 1}\\
	A
	\arrow{r}{\sigma}
	&A\T A
	\arrow{r}{\alpha\T 1}
	&\1\T A
	\end{tikzcd}$$
	commutes, we also get $e\eta \alpha=e$. 
	Suppose $e'$ is also an $A$-linear morphism with  	$\alpha e'=\alpha$ and $e'\eta \alpha=e'$. 
	Then,
	$e=e\eta\alpha=e\eta\alpha e'=ee'=e'e=e'\eta\alpha e=e'\eta\alpha=e'$ by~Remark~\ref{rem:endo}.
	For (b) let $1\leq i,j \leq n$. From the commuting diagram
	 $$\begin{tikzcd}[column sep=large]
	 A
	 \arrow{r}{\alpha_i}
	 \arrow{d}{1\T \eta}
	 & \1
	 \arrow{r}{\eta}
	 & A
	 \arrow{d}{e_j}\\
	 A\T A
	 \arrow{r}{1\T e_j}
	 \arrow{d}{\mu}
	 & A \T A
	 \arrow{d}{\mu}
	 \arrow{r}{\alpha_i\T 1}
	 & A
	 \arrow{d}{\alpha_i}\\
	 A\arrow{r}{e_j}
	 & A \arrow{r}{\alpha_i} 
	 & \1,
	 \end{tikzcd}$$
	we see that $\alpha_i e_j \eta \alpha_i=\alpha_i e_j$.
	Hence, $(\alpha_i e_j \eta) (\alpha_i e_j \eta)
	 =\alpha_i e_j e_j \eta=\alpha_i e_j \eta$,
	so the morphism $\alpha_i e_j \eta:\1\rightarrow \1$ is idempotent 
	and equals $0$ or $1_{\1}$.
	In the first case, $\alpha_i e_j =\alpha_i e_j \eta \alpha_j =0$ and $e_i e_j=e_i \eta\alpha_i e_j=0$, in particular $i\neq j$.
	On the other hand, if $\alpha_i e_j \eta=1_{\1}$ we get
	$\alpha_i e_j=\alpha_i e_j \eta \alpha_i=\alpha_i$ and $\alpha_i e_j=\alpha_i e_j \eta \alpha_j=\alpha_j$, so~$i=j$.
	\end{proof}



\begin{lemma}\label{lem:ringtoBB}
	Let $(A,\mu_A,\eta_A)$ and $(B,\mu_B,\eta_B)$ be separable rings in~$\K$.
	\begin{enumerate}[label=(\alph*)] 
		\item Suppose $f:A\rightarrow B$ and $g:B\rightarrow A$ are ring morphisms such that $g f=1_A$. We equip $A$ with the structure of $B^e$-module via the morphism $g$. There exists a $B^e$-linear morphism $\widetilde{f}:A\rightarrow B$ such that $g \widetilde{f}=1_A$.
		In particular, $A$ is a direct summand of $B$ as a $B^e$-module.
		
		\item Suppose $A$ is indecomposable. Let $g_i:B\rightarrow A$ be distinct ring morphisms for $1\leq i \leq n$ and suppose $f:A\rightarrow B$ is a ring morphism with $g_i f=1_A$. Then $A^{\oplus n}$ is a direct summand of $B$ as a $B^e$-module, with projections $g_i: B\rightarrow A$ for~$1\leq i \leq n$.
		
	\end{enumerate}
\end{lemma}

\begin{proof}
	Considering the $A$-module structure on $B$ given by $f$, we note that $g:~B\rightarrow~A$ is $A$-linear:
	$$\begin{tikzcd}
	A\T B
	\arrow{r}{f\T 1}
	\arrow[-, double equal sign distance]{d}
	& B\T B
	\arrow{r}{\mu_B}
	\arrow{d}{g\T g}
	& B
	\arrow{d}{g}
	\\
	A\T B
	\arrow{r}{1\T g}
	& A \T A
	\arrow{r}{\mu_A}
	& A.
	\end{tikzcd}$$
	We can thus apply~Lemma~\ref{lem:idempotent} to the ring morphism $\bar{g}:\overline{B}\rightarrow \1_A$ in $\M A$ and find an idempotent $\overline{B}^e$-linear morphism $\bar{e} :\overline{B}\rightarrow \overline{B}$ such that $\bar{g}\bar{e}=\bar{g}$ and $\bar{e}\eta_{\bar{B}}\bar{g}=\bar{e}$.  
	Forgetting the $A$-action, $U_A(\bar{e}):=e:B\rightarrow B$ is idempotent and $B^e$-linear, with $ge=g$ and $efg=e$. Let $\widetilde{f}:=ef$.
	We need to show that $\widetilde{f}$ is $B^e$-linear, where $B^e$ acts on $A$ via $g$.
	Left $B$-linearity of $\widetilde{f}$ follows from the commuting diagram
	$$\begin{tikzcd}
	B\T A
	\arrow{r}{g\T 1}
	\arrow{dd}{1\T f}
	& A\T A
	\arrow{r}{\mu_A}
	\arrow{d}{f\T f}
	& A
	\arrow{d}{f}
	\\
	& B\T B
	\arrow{r}{\mu_B}
	\arrow{d}{e\T 1}
	& B
	\arrow{d}{e}
	\\
	B\T B 
	\arrow{r}{e\T 1}
	\arrow[-, double equal sign distance]{d}
	& B\T B
	\arrow{r}{\mu_B}
	& B
	\arrow[-, double equal sign distance]{d}\\
	B\T B 
	\arrow{r}{1\T e}
	& B\T B
	\arrow{r}{\mu_B}
	& B,\\
	\end{tikzcd}$$
	and right $B$-linearity follows similarly. 
	Finally, $g \widetilde{f}=gef=gf=1_A$.
	
	For (b), let $g_i:B\rightarrow A$ be distinct ring morphisms with $g_i f=1_A$ for~$1\leq i \leq n$. As in part (a), we find idempotent $B^e$-linear morphisms $e_i:B\rightarrow B$ and $B^e$-linear morphisms $\widetilde{f_i}:=e_i f$ with $g_i \widetilde{f_i}=1_A$ and $e_i=\widetilde{f_i} g_i$. In fact, Lemma~\ref{lem:idempotent}(b) shows the $e_i$ are orthogonal. Seeing how $A=\im(e_i)$, we conclude $A^{\oplus n}$ is a direct summand of $B$ as a $B^e$-module, with projections $g_i: B\rightarrow A$ for~$1\leq i \leq n$.
\end{proof}

\begin{corollary}\label{cor:BringfactorF_A(B)}
	Let $A$ and $B$ be separable rings in~$\K$ and suppose $B$ is an $A$-algebra. The corresponding ring $\overline{B}$ in $\M A$ is a ring factor of $F_A(B)$.
\end{corollary}

\begin{proof}
	Applying~Lemma~\ref{lem:ringtoBB} to the ring morphisms $f: B\xrightarrow{\eta_A\T 1_B} A\T B$ and $g$ given by the action of $A$ on $B$, we see that $B$ is a direct summand of $A\T B$ as $(A\T B)^e$-modules in $\K$. In particular, $\overline{B}$ is a direct summand of $F_A(B)$ as $F_A(B)^e$-modules in $\M A$. By~Lemma~\ref{lem:AAtoring}, $\overline{B}$ admits a ring structure under which $\overline{B}$ becomes a ring factor of $F_A(B)$. This new ring structure on $\overline{B}$ is the original one, seeing how the projection $g:F_B(A)\rightarrow \overline{B}$ is a ring morphism for both structures. 
\end{proof}

\section{Degree of a Separable Ring}

We recall Balmer's definition of the degree of a separable ring in a tensor-triangulated category, see~\cite{bdegree}, and show the definition works for any idempotent-complete symmetric monoidal category~$\K$.

\begin{theorem}\label{th:splittingtheorem}
	Let $A$ and $B$ be separable rings in~$\K$. 
	Suppose $f:A\rightarrow B$ and $g:B\rightarrow A$ are ring morphisms such that $g f=1_A$. There exists a separable ring $C$ in $\K$ and a ring isomorphism $h:B\xrightarrow{\sim} A\times C$ such that $\pr_1 h=g$.
	If we equip $C$ with the $A$-algebra structure coming from $\pr_2 h f$, it is unique up to isomorphism of~$A$-algebras.
\end{theorem}

\begin{proof}
	This proposition is proved in~\cite[Th.2.4]{bdegree} when $\K$ is a tensor-triangulated category.
	In our case, Lemma~\ref{lem:ringtoBB} yields an isomorphism $h:B\xrightarrow{\sim} A\oplus C$ of $B^e$-modules with $\pr_1 h=g$. By~Lemma~\ref{lem:AAtoring}, $A$ and $C$ admit ring structures under which $h$ becomes a ring isomorphism. This new ring structure on $A$ is the original one, seeing how $1_A:A\xrightarrow{f}B\xrightarrow{\pr_1 h} A$ is a ring morphism. 
	The rest of the proof is identical to the proof in~\cite[Th.2.4]{bdegree}.
\end{proof}

\begin{definition}(\cite[Def.3.1]{bdegree}).
	Let $(A, \mu, \eta)$ be a separable ring in~$\K$.
	Applying~Theorem~\ref{th:splittingtheorem} to the ring morphisms $f=1_A\T \eta:A\rightarrow A\T A$ and $g=\mu:A\T A\rightarrow A$, we find a separable $A$-algebra $A'$, unique up to isomorphism, and a ring isomorphism $h:A\T A\xrightarrow{\sim} A\times A'$ such that $\pr_1 h=\mu$. 
 
	The \emph{splitting tower} 
	$$\1=A^{[0]}\xrightarrow{\eta} A=A^{[1]}\rightarrow A^{[2]}
	\rightarrow \ldots \rightarrow A^{[n]}\rightarrow A^{[n+1]}\rightarrow \ldots$$
	is defined inductively by $A^{[n+1]}=(A^{[n]})'$, where we consider $A^{[n]}$ as a ring in $\M {A^{[n-1]}}$. We say the \emph{degree} of $A$ is $d$, writing $\deg_{\K}(A)=d$,
	if $A^{[d]}\neq 0$ and $A^{[d+1]}=0$. We say $A$ has \emph{infinite degree} if $A^{[d]}\neq 0$ for all $d\geq 0$.
\end{definition}

\begin{remark}
	By construction, we have $(A^{[n]})^{[m+1]}\cong A^{[n+m]}$ as $A^{[n+m-1]}$-algebras for all $m\geq 0$ and $n\geq 1$, where we regard $A^{[n]}$ as a ring in $\M {A^{[n-1]}}$.
	In other words, $\underset{\M {A^{[n-1]}}}\deg \! \!\!(A^{[n]})=\deg_{\K}(A)-n+1$ for $1\leq n\leq \deg_{\K}(A)+1$.
\end{remark}

\begin{example}\label{ex:etale}
	Let $R$ be a commutative ring and suppose $A$ is a commutative projective separable $R$-algebra.  
	If $\Spec R$ is connected, then the degree of $A$ as a ring object in the categories $R-\!\!\operatorname{mod}$ and $\D(R)$ (as in~Example~\ref{ex:etale}) recovers its rank as an $R$-module. 
\end{example}

\begin{proposition}\label{prop:degreeproperties}
	Let $A$ and $B$ be separable rings in~$\K$.
	\begin{enumerate}[label=(\alph*)]
		\item  We have $F_{A^{[n]}}(A)\cong \1_{A^{[n]}}^{\times n}\times A^{[n+1]}$ as $A^{[n]}$-algebras.

		\item Let $F:\K\rightarrow \cat L$ be an additive monoidal functor. For every $n\geq 0$, the rings $F(A^{[n]})$ and $F(A)^{[n]}$ are isomorphic. In particular, $\deg_{\cat L}(F(A))\leq~\deg_{\K} (A)$.

		\item  Suppose $A$ is a $B$-algebra. Then 
		$\deg_{\M B}(F_B(A))= \deg_{\K} (A)$.
	\end{enumerate}
\end{proposition}

\begin{proof}
	The proofs for (a) and (b) in~\cite[Th.3.7, 3.9]{bdegree} still hold in our (not-necessarily triangulated) setting. To prove (c), note that $A^{[n]}$ is a $B$-algebra and hence a direct summand of $F_B(A^{[n]})\cong F_B(A)^{[n]}$. This means $F_B(A)^{[n]}\neq 0$ when $A^{[n]}\neq 0$ so that $\deg_{\M B}(F_B(A))\geq \deg_{\K} (A)$.
\end{proof}

\begin{lemma}\emph{(\cite[Lem.3.11]{bdegree}).}\label{lem:1degree}
	Let $n\geq 1$ and $A:=\1^{\times n}\in\K$. There is an isomorphism $A^{[2]}\cong A^{\times (n-1)}$ of $A$-algebras.
\end{lemma}

\begin{proof}
	We prove there is an $A$-algebra isomorphism $\lambda:A\T A\xrightarrow{\sim} A\times A^{\times (n-1)}$ with $\pr_1 \lambda=\mu_A$.
	We write $A=\prod_{i=0}^{n-1} \1_i$,
	$A\T A=\prod_{0\leq i,j\leq n-1} \1_i\T \1_j$
	and $A^{\times n}= \prod_{k=0}^{n-1} \prod_{i=0}^{n-1} \1_{ik}$ with $\1=\1_i=\1_{ik}$ for all $i,k$.
	Define $\lambda:A\T A\rightarrow A^{\times n}$ by mapping the factor $\1_i\T \1_j$ identically to $\1_{i(i-j)}$, with indices in~$\Z_n$. Then, $\lambda$ is an $A$-algebra isomorphism and $\pr_{k=0}\lambda=\mu_A$.
\end{proof}

\begin{corollary}\label{cor:1degree}
	Let $n\geq 1$. Then $\deg_{\K} (\1^{\times n})=n$ and $(\1^{\times n})^{[n]}\cong\1^{\times n!}$ in~$\K$.
\end{corollary}

\begin{proof}
	Let $A:=\1^{\times n}$. The result is clear when $n=1$, and we proceed by induction on~$n$. By~Lemma~\ref{lem:1degree}, we know $A^{[2]}\cong \1_A^{\times (n-1)}$ in $\M A$. Assuming the induction hypothesis, $\deg_{\M A}(A^{[2]})=n-1$ and $$A^{[n]}\cong(A^{[2]})^{[n-1]}\cong\1_A^{\times (n-1)!}\cong(\1^{\times n})^{\times(n-1)!}\cong\1^{\times n!}.$$
\end{proof}

\begin{lemma}\label{cor:1factor}
	Let $A$ and $B$ be separable rings of finite degree in $\K$. Then,
	\begin{enumerate}[label=(\alph*)]
		\item $\deg (A\times B)\leq \deg(A)+\deg(B)$
		\item $\deg(A\times \1^{\times n})= \deg(A)+n$
		\item $\deg(A^{\times t})=\deg(A) \cdot t$.
	\end{enumerate}
\end{lemma}

\begin{proof}
	To prove (a), let $n:=\deg(A\times B)$ and $C:=(A\times B)^{[n]}$. 
	Writing $A':=F_{C}(A)$ and $B':= F_{C}(B)$, we know from~Proposition~\ref{prop:degreeproperties}(a) that
	$$A'\times B'\cong\1_{C}^n.$$ 
	If we let $D:=(A')^{[\deg(A')]}$ and
	apply $F_D$ to the isomorphism, we get 
	$$\1_{D}^{\deg(A')}\times F_{D}(B')\cong \1_{D}^n.$$
	Similarly, putting $E:=(F_D(B'))^{[\deg(F_D(B'))]}$ and applying $F_E$ gives
	$$\1_{E}^{\deg(A')}\times \1_{E}^{\deg(F_D(B'))}\cong \1_{E}^n.$$
	This shows $n=\deg(A')+\deg(F_D(B'))\leq \deg(A)+\deg (B)$ by~Proposition~\ref{prop:degreeproperties}(b).
	For (b), let $B:=A^{[\deg(A)]}$. Then,
	$F_B(A\times\1^{\times n})\cong \1_B^{\times \deg(A)}\times \1_B^{\times n}$
	and we find $$\deg(A\times \1^{\times n})\geq\deg(F_B(A\times \1^{\times n}))=\deg (A)+n.$$
	To prove (c), we write $B:=A^{[\deg(A)]}$ again and note that
	$F_B(A^{\times t})\cong \left(\1_{B}^{\times\deg(A)}\right)^{\times t}$.
	Hence, $\deg(A^{\times t})\geq \deg(F_B(A^{\times t}))=\deg(A)\cdot t$.
\end{proof}

\section{Counting Ring Morphisms}
.
\begin{lemma}\label{lem:nrmorphisms}
	Let $A$ be a separable ring in~$\K$ and suppose $\1$ is indecomposable. If there are $n$ distinct ring morphisms $A\rightarrow \1$, then $A$ has $\1^{\times n}$ as a ring factor. In particular, there are at most $\deg A$ distinct ring morphisms $A\rightarrow \1$.
\end{lemma}

\begin{proof}
	Let $\alpha_i:A\rightarrow \1$ be distinct ring morphisms for $1\leq i\leq n$.
	By~Lemma~\ref{lem:ringtoBB}(b), we know that $\1^{\oplus n}$ is a direct summand of $A$ as an $A^e$-module, with projections $\alpha_i: A\rightarrow \1$ for~$1\leq i \leq n$. Moreover, Lemma~\ref{lem:AAtoring} shows that every such summand~$\1$ admits a ring structure, under which $\1^{\times n}$ becomes a ring factor of $A$ and the projections $\alpha_i$ are ring morphisms. In fact, these new ring structures on $\1$ are the original one, seeing how $\alpha_i\eta=1_{\1}$ is a ring morphism for every~$1\leq i \leq n$.
	Finally,~Corollary~\ref{cor:1factor}(b) shows that~$\deg(A)\geq n$.
\end{proof}

\begin{proposition}\label{prop:An}
	Let $A$ and $B$ be separable rings in $\K$ and suppose $B$ is indecomposable. Let $n\geq 1$. The following are equivalent:
	\begin{enumerate}[label=(\roman*)]
		\item There are (at least) $n$ distinct ring morphisms $A\rightarrow B$ in~$\K$.
		\item The ring $\1_B^{\times n}$ is a factor of $F_B(A)$ in $\M B$.
		\item There is a ring morphism $A^{[n]}\rightarrow B$ in~$\K$.
	\end{enumerate}
\end{proposition}

\begin{proof}
	Firstly, we claim there is a one-to-one correspondence between ring morphisms $\alpha:A\rightarrow B$ in $\K$ and ring morphisms $\beta: F_B(A) \rightarrow \1_B$ in~$\M B$. Indeed, this correspondence sends $\alpha: A\rightarrow B$ in $\K$ to the $B$-algebra morphism 
	$$B\T A \xrightarrow{1_B \T \alpha}B\T B\xrightarrow{\mu} B,$$
	and conversely, $\beta: F_B(A) \rightarrow \1_B$ gets mapped to
	$A\xrightarrow{\eta_B\T 1_A}
	 B\T A \xrightarrow{\beta} B$ in~$\K$.

	To show (i)$\Rightarrow$(ii), note that $n$ distinct ring morphisms $A\rightarrow B$ in~$\K$ give $n$ distinct ring morphisms $F_B(A)\rightarrow \1_B$ in $\M{B}$. 
	By~Lemma~\ref{lem:nrmorphisms}, $\1_B^{\times n}$ is a ring factor of~$F_B(A)$. 
	For (ii)$\Rightarrow$(i), suppose $\1_B^{\times n}$ is a ring factor of $F_B(A)$ in $\M B$ and consider the projections $\pr_i:F_B(A)\rightarrow \1_B$ with $1\leq i\leq n$. By the claim, there are at least $n$ distinct ring morphisms $A\rightarrow B$ in~$\K$.

	We show (ii)$\Rightarrow$(iii) by induction on $n$.
	The case $n=1$ has already been proven.
	Let $n\geq 1$ and suppose $\1_B^{\times (n+1)}$ is a ring factor of $F_B(A)$. 
	By the induction hypothesis, there exists a ring morphism $A^{[n]}\rightarrow B$. As usual, we write $\overline{B}$ for the separable ring in $\M {A^{[n]}}$ corresponding to the $A^{[n]}$-algebra~$B$ in~$\K$. 
	The diagram
	\begin{equation}\label{eq:An}
	\begin{tikzcd}
	\K
	\arrow{r}{F_{A^{[n]}}}
	\arrow{d}{F_B} 
	& \M {A^{[n]}}
	\arrow{d}{F_{\overline{B}}}\\
	\M B
	\arrow{r}{\simeq}
	& \Mod {\M {A^{[n]}}}{\overline{B}}
	\end{tikzcd}
	\end{equation}
	from~Proposition~\ref{prop:diagramequivalence} 
	shows that
	$F_B(A)$ is mapped to $F_{\overline{B}}(F_{A^{[n]}}(A))$ under the equivalence
	$\M B\simeq \Mod {\M {A^{[n]}}}{\overline{B}}$.
	It follows that $\1_{\overline{B}}^{\times (n+1)}$ is a ring factor of~$F_{\overline{B}}(F_{A^{[n]}}(A))$.
	On the other hand, by~Proposition~\ref{prop:degreeproperties}(a) we know that
	\begin{equation}\label{eq:FBFA}
	F_{\overline{B}}(F_{A^{[n]}}(A))
	\cong F_{\overline{B}}(\1_{A^{[n]}}^{\times n}\times A^{[n+1]})\cong \1_{\overline{B}}^{\times n}\times F_{\overline{B}}(A^{[n+1]}).
	\end{equation}
	Hence, $\1_{\overline{B}}$ is a ring factor of $F_{\overline{B}}(A^{[n+1]})$ by~Proposition~\ref{prop:uniquedecomposition} and we conclude there exists a ring morphism  $A^{[n+1]}\rightarrow \overline{B}$ in $\M {A^{[n]}}$.

	To show (iii)$\Rightarrow$(ii), suppose $B$ is an $A^{[n]}$-algebra and write $\overline{B}$ for the corresponding separable ring in $\M {A^{[n]}}$.
	Using diagram~\ref{eq:An} again,
	it is enough to show that $\1_{\overline{B}}^{\times n}$ is a ring factor of $F_{\overline{B}}(F_{A^{[n]}}(A))$. This follows from~\ref{eq:FBFA}.
\end{proof}

\begin{theorem}\label{th:nrmorphisms}
	Let $A$ and $B$ be separable rings in $\K$, where $A$ has finite degree and $B$ is indecomposable.
	There are at most $\deg(A)$ distinct ring morphisms from $A$ to~$B$. 
\end{theorem}

\begin{proof}
	If there are $n$ distinct ring morphisms from $A$ to~$B$, we know $\1_B^{\times n}$ is a ring factor of $F_B(A)$ by~Proposition~\ref{prop:An}. So,
	$n\leq \deg_{\M{B}}(F_B(A))\leq \deg_{\K}(A)$ by~Proposition~\ref{prop:degreeproperties}(b) and~Corollary~\ref{cor:1factor}(b).
\end{proof}

\begin{remark}
	The assumption $B$ is indecomposable is necessary in~Theorem~\ref{th:nrmorphisms}. Indeed, $\deg(\1^{\times{n}})=n$ but $\1^{\times{n}}$ has at least $n!$ ring endomorphisms.
\end{remark}


\section{Quasi-Galois Theory}

Suppose $(A,\mu, \eta)$ is a nonzero ring in $\K$ and $\Gamma$ is a finite set of ring endomorphisms of $A$ with $1_A\in\Gamma$. Consider the ring $\prod_{\gamma\in \Gamma} A_{\gamma}$, where we write $A_{\gamma}=A$ for all $\gamma \in \Gamma$ to keep track of the different copies of~$A$.
We define ring morphisms $\varphi_1:A\rightarrow \prod_{\gamma\in \Gamma} A_{\gamma}$ 
by $\pr_{\gamma}\varphi_1=1_A$ 
and $\varphi_2:A\rightarrow \prod_{\gamma\in \Gamma} A_{\gamma}$ by $\pr_{\gamma}\varphi_2=\gamma$ for all $\gamma \in \Gamma$.
Thus, $\varphi_1$ renders the (standard) left $A$-algebra structure on $\prod_{\gamma\in \Gamma} A_{\gamma}$ and we introduce a right
$A$-algebra structure on $\prod_{\gamma\in \Gamma} A_{\gamma}$ via $\varphi_2$. 

\begin{definition}\label{def:lambda}
	We will consider the following ring morphism:
	$$\begin{tikzcd}\lambda_{\Gamma}=\lambda: \hspace{0.1 cm}
	 A\T A \arrow{r} & \prod_{\gamma\in \Gamma} A_{\gamma}\end{tikzcd}\hspace{0.5 cm}
	\text{with}\hspace{0.1 cm}
	\pr_{\gamma} \lambda=\mu (1 \T \gamma).$$
	Note that $\lambda(1 \T \eta)=\varphi_1$ and $\lambda(\eta \T 1)=\varphi_2$,
	\begin{equation}\label{eq:lambdatriangle}
	\begin{tikzcd}[column sep=small]
	&& A 
	\arrow[swap, yshift=0.5ex, xshift=-0.5ex]{lld}{1\T \eta}
	\arrow[ yshift=-0.3ex, xshift=0.7ex]{lld}{\eta\T 1}
	\arrow[ yshift=0.5ex, xshift=0.5ex]{rrd}{\varphi_2}
	\arrow[swap,yshift=-0.3ex, xshift=-0.7ex]{rrd}{\varphi_1}&&\\
	A\T A
	\arrow{rrrr}{\lambda}
	&&&& \prod_{\gamma\in \Gamma} A_{\gamma} ,
	\end{tikzcd}\end{equation}
	so that $\lambda$ is an $A^e$-algebra morphism. 
\end{definition}

\begin{lemma}\label{lem:lambdaiso}
	Suppose $\lambda_{\Gamma}:A\T A \rightarrow \prod_{\gamma\in \Gamma} A_{\gamma}$ is an isomorphism.
	\begin{enumerate}[label=(\alph*)]
	
		\item There is an $A^e$-linear morphism $\sigma: A \rightarrow A\T A$ such that $\mu (1 \T \gamma)\sigma =\delta_{1,\gamma}$ for every $\gamma \in \Gamma$. In particular, $A$ is separable.


		\item Let $\gamma \in \Gamma$. If there exists a nonzero ring $B$ in $\K$ and ring morphism $\alpha:A\rightarrow B$ with $\alpha \gamma= \alpha$, then $\gamma=1$.

		\item The ring $A$ has degree $|\Gamma|$ in~$\K$.

	\end{enumerate}
\end{lemma}

\begin{proof}
	To prove~(a), consider the $A^e$-linear morphism $\sigma:=\lambda^{-1} \inc_1 : A \rightarrow A\T A$. The following diagram shows that $\mu (1 \T \gamma)\sigma =\delta_{1,\gamma}$:
	$$\begin{tikzcd} 
	A \arrow{rr}{\sigma} \arrow[hook, swap]{rd}{\inc_1}
	&& A\T A  \arrow{r}{1\T \gamma} \arrow[swap]{rd}{\lambda}
	& A\T A  \arrow{r}{\mu}
	& A.\\
	& \prod_{\gamma\in \Gamma} A_{\gamma}  \arrow[swap]{ru}{\lambda^{-1}}
	&&\prod_{\gamma\in \Gamma} A_{\gamma}  \arrow[two heads, swap]{ru}{\pr_\gamma}
	\end{tikzcd}$$
	For~(b), suppose $\alpha\gamma= \alpha$ and 
	$\sigma: A \rightarrow A\T A$ as in (a). We get 
	$$\alpha=\alpha\mu \sigma
	=\mu(\alpha\T \alpha)\sigma
	=\mu (\alpha\T \alpha)(1\T \gamma)\sigma
	=\alpha\mu (1\T \gamma)\sigma
	=\alpha \delta_{\gamma,1}.$$
	Hence, $\alpha=0$ or $\gamma=1_A$.
	Finally, given that 
	$F_A(A)\cong \1_A^{\times |\Gamma|}$ in $\M A$, Proposition~\ref{prop:degreeproperties}(c) shows that $\deg(A)~=~|\Gamma|$.
\end{proof}

\begin{definition}\label{def:galois}
	Suppose $A$ is a nonzero ring in $\K$ and $\Gamma$ is a finite group of ring automorphisms of $A$.
	We say $A$ is \emph{quasi-Galois in} $\K$ with group $\Gamma$ if $\lambda_{\Gamma}:A\T A \rightarrow \prod_{\gamma\in\Gamma}A_{\gamma}$ is an isomorphism. 
	We also call
	$F_A: \K\longrightarrow \M A$ a \emph{quasi-Galois extension} with group~$\Gamma$.
\end{definition} 

\begin{example}
	Let $A:=\1^{\times{n}}$ and consider the ring morphism $\gamma:=(1 2\cdots n)$ which permutes the factors. Then $A$ is quasi-Galois with group $\Gamma=\{\gamma^i\mid 0\leq i\leq n-1\}\cong \Z_n$.
	Indeed, the isomorphism $\lambda:A\T A\rightarrow A^{\times n}$ constructed in the proof of~Lemma~\ref{lem:1degree} is exactly~$\lambda_{\Gamma}$.
	In particular, $\Gamma$ does not always contain all ring automorphisms of~$A$.
\end{example}

\begin{example}\label{ex:etalegalois}
	Let $R$ be a commutative ring, $A$ a commutative $R$-algebra and $\Gamma$ a finite group of ring automorphisms of $A$ over $R$. Suppose $A$ is a Galois extension of $R$ relative to $\Gamma$ in the sense of~\cite[App.]{AG}.
	In particular, $A$ is projective and separable as an $R$-module. Then the ring $A$ in $\K=\D(R)$ is quasi-Galois with group~$\Gamma$.
\end{example}

\begin{lemma}\label{lem:qgextensions}
	Let $A$ be quasi-Galois of degree $d$ in $\K$ with group $\Gamma$ and suppose $F:\K\rightarrow~\cat L$ is an additive monoidal functor. If $F(A)\neq 0$, then $F(A)$ is quasi-Galois of degree $d$ in $\cat L$ with group $F(\Gamma)=\{F(\gamma)\mid \gamma \in \Gamma\}$.
	In particular, being quasi-Galois is stable under extension-of-scalars.
\end{lemma}

\begin{proof}
	We immediately see that
	$$F(\lambda_{\Gamma}): F(A)\T F(A)\cong F(A\T A) \longrightarrow \prod_{\gamma\in \Gamma}F(A)$$ 
	is an isomorphism in $\cat L$, so it suffices to show $\Gamma\cong F(\Gamma)$ and $F(\lambda_{\Gamma})=\lambda_{F(\Gamma)}$.
	Now, $\lambda_\Gamma$ is defined by $\pr_{\gamma}\lambda_\Gamma=\mu_A(1_A\T \gamma)$, 
	hence $\pr_{\gamma}F(\lambda_\Gamma)=\mu_{F(A)}(1_{F(A)}\T F(\gamma))$ for every $\gamma\in \Gamma$.
	In particular, the morphisms $\mu_{F(A)}(1_{F(A)}\T F(\gamma))$ with $\gamma\in \Gamma$ are distinct. This shows the morphisms $F(\gamma)$ with $\gamma\in \Gamma$ are distinct, so that $\Gamma\cong F(\Gamma)$ and $F(\lambda_{\Gamma})=\lambda_{F(\Gamma)}$.
\end{proof}

\begin{proposition}\label{prop:actionofgamma}
	Suppose $A$ is quasi-Galois in~$\K$ with group~$\Gamma$.

	\begin{enumerate}[label=(\alph*)]
		\item If $B$ is a separable indecomposable $A$-algebra,
		then $\Gamma$ acts faithfully and transitively on the set of ring morphisms from $A$ to~$B$. In particular, there are exactly $\deg(A)$ distinct ring morphisms from $A$ to~$B$ in~$\K$. 

		\item If $A$ is indecomposable then $\Gamma$ contains all ring endomorphisms of~$A$.
	\end{enumerate}
\end{proposition}

\begin{proof}
	Note that the set $S$ of ring morphisms from $A$ to $B$ is non-empty and $\Gamma$ acts on $S$ by precomposition.
	The action is faithful by~Lemma~\ref{lem:lambdaiso}(b) and transitive because $|S|\leq\deg A=|\Gamma|$ by~Theorem~\ref{th:nrmorphisms}. 
	In particular, $A$ has exactly $\deg A=|\Gamma|$ ring endomorphisms in~$\K$.
\end{proof}

By the above proposition, we can simply say an indecomposable ring $A$ in $\K$ is quasi-Galois, with the understanding that the Galois-group $\Gamma$ contains all ring endomorphisms of~$A$.

\begin{theorem}\label{th:galoisequivalences}
	Let $A$ be a separable indecomposable ring of finite degree in $\K$ and write $\Gamma$ for the set of ring endomorphisms of~$A$. The following are equivalent:
	\begin{enumerate}[label=(\roman*)]
		\item $|\Gamma|=\deg(A)$.
		\item $F_A(A)\cong \1_A^{\times t}$ in $\M A$ for some $t> 0$.
		\item $\lambda_{\Gamma}: A\T A\rightarrow \prod_{\gamma\in \Gamma} A_{\gamma}$ is an isomorphism.
		\item $\Gamma$ is a group and $A$ is quasi-Galois in~$\K$ with group~$\Gamma$.
	\end{enumerate}
\end{theorem} 

\begin{proof}
	First note that $d:=\deg(A)=\deg(F_A(A))$ by~Proposition~\ref{prop:degreeproperties}(c).
	To show (i)$\Rightarrow$(ii), recall that $\1_A^{\times d}$ is a ring factor of $F_A(A)$ if $|\Gamma|=d$ by~Lemma~\ref{prop:An}. By~Corollary~\ref{cor:1factor}(b), we know $F_A(A)\cong~\1_A^{\times d}$.
	For (ii)$\Rightarrow$(iii), we note that $t=d$ and consider an $A$-algebra isomorphism $l:A\T A\xrightarrow{\simeq} A^{\times d}$.
	We define ring endomorphisms
	$$\begin{tikzcd}
	\alpha_i: A\arrow{r}{\eta\T 1_A}
	& A\T A \arrow{r}{l}
	& A^{\times d} \arrow{r}{\pr_i}
	& A,
	\end{tikzcd}
	\hspace{1.5 cm} i=1,\ldots,d,$$
	such that
	$\mu(1_A\T \alpha_i)=\pr_i l (\mu\T 1_A)(1_A\T \eta \T 1_A)=\pr_i l$ 
	for every~$i$. This shows the $\alpha_i$ are all distinct, so that
	$\Gamma=\{\alpha_i\mid 1\leq i\leq d\}$ by~Theorem~\ref{th:nrmorphisms} and $l=\lambda_{\Gamma}$.
	For (iii)$\Rightarrow$(iv), we show that
	every $\gamma \in \Gamma$ is an automorphism.
	By~Lemma~\ref{lem:lambdaiso} (a), we can find an $A^e$-linear morphism $\sigma:A\rightarrow A\T A$ such that $\mu (1 \T \gamma)\sigma =\delta_{1,\gamma}$ for every~$\gamma\in \Gamma$.
	Let $\gamma \in \Gamma$ and note that $\gamma=\mu (\gamma\T 1)(1\T\gamma)\sigma$ so that 
	$(1\T \gamma)\sigma:A\rightarrow A\T A$ is nonzero.
	Thus there exists $\gamma'\in \Gamma$ such that 
	$$\pr_{\gamma'} \lambda_{\Gamma} (1\T \gamma)\sigma=\mu (1\T\gamma')(1\T \gamma)\sigma =\delta_{1,\gamma' \gamma}$$ is nonzero. This means $1=\gamma'\gamma $ and 
	$\gamma' (\gamma \gamma')=\gamma'$ so $\gamma \gamma'=1$ by~Lemma~\ref{lem:lambdaiso}(b).
	Finally, (iv)$\Rightarrow$(i) is the last part of~Lemma~\ref{lem:lambdaiso}.
\end{proof}

\begin{corollary}\label{cor:galoisfactors}
	Let $A, B$ and $C$ be separable rings in $\K$ with $A\cong B\times C$, and suppose $B$ is indecomposable.
	If $F_A(A)\cong \1_A^{\times d}$, then $B$ is quasi-Galois. 
	In particular, being quasi-Galois is stable under passing to indecomposable factors. 
\end{corollary}

\begin{proof}
	Consider the decomposition
	$\M A\cong \M{B}\times \M{C}$,
	under which $F_A(A)$ corresponds to $(F_B(B\times C), F_C(B\times C))$ and $\1_A^{\times d}$ corresponds to~$(\1_B^{\times d} , \1_C^{\times d})$.
	Given that $\1_B$ is indecomposable and $F_B(B)$ is a ring factor of $\1_B^{\times d}$ in $\M B$, we know 
	$F_B(B)\cong \1_B^{\times t}$ for some $1\leq t\leq d$. The result now follows from~Theorem~\ref{th:galoisequivalences}.
\end{proof}



\section{Splitting Rings}

\begin{definition}\label{def:splittingring}
	Let $A$ and $B$ be separable rings of finite degree in~$\K$. We say $B$ \emph{splits} $A$ if $F_B(A)\cong \1_B^{\times \deg(A)}$ in $\M B$. 
	We call an indecomposable ring $B$ a \emph{splitting ring} of $A$ if $B$ splits $A$ and any ring morphism $C\rightarrow B$, where $C$ is an indecomposable ring splitting~~$A$, is an isomorphism. 
\end{definition}

\begin{remark}\label{rem:BsplitsA}
	Let $A$ be a separable ring in $\K$ with $\deg(A)=d$. The ring $A^{[d]}$ in $\K$ splits $A$ by~Proposition~\ref{prop:degreeproperties}(a).
	Moreover, if $B$ is a separable indecomposable ring in $\K$, then $B$ splits $A$ if and only if $B$ is an $A^{[d]}$-algebra.
	This follows immediately from~Proposition~\ref{prop:An}.
\end{remark}

\begin{remark}\label{rem:Adsplits}
	Let $A$ be a separable ring in $\K$ with $\deg(A)=d$. The ring $A^{[d]}$ in $\K$ splits itself by~Proposition~\ref{prop:degreeproperties}(a),(b) and~Corollary~\ref{cor:1degree}:
	$$F_{A^{[d]}}(A^{[d]})\cong (F_{A^{[d]}}(A))^{[d]}\cong (\1_{A^{[d]}}^{\times d})^{[d]}\cong \1_{A^{[d]}}^{\times d!}.$$ 
\end{remark}

\begin{lemma}\label{lem:selfsplitfactors}
	Let $A$ be a separable ring in $\K$ that splits itself. If $A_1$ and $A_2$ are indecomposable ring factors of~$A$, then
	any ring morphism $A_1\rightarrow A_2$ is an isomorphism.
\end{lemma}

\begin{proof}
	Let $A_1$ and $A_2$ be indecomposable ring factors of $A$ and suppose there is a ring morphism $f:A_1\rightarrow A_2$.
	We know $F_{A_1}(A)\cong \1_{A_1}^{\times \deg(A)}$ because $A$ splits itself.
	Meanwhile, $F_{A_1}(A_2)$ is a ring factor of~$F_{A_1}(A)$, so that 
	$F_{A_1}(A_2)\cong \1_{A_1}^{\times d}$ for some $d\geq 0$. In fact, $d=\deg(A_2)\geq 1$ by~Proposition~\ref{prop:degreeproperties}(c).
	Proposition~\ref{prop:An} shows there exists a ring morphism $g: A_2\rightarrow A_1$.
	Note that $A_1$ and $A_2$ are quasi-Galois by~Corollary~\ref{cor:galoisfactors}, so that the ring morphisms $gf:A_1 \rightarrow A_1$ and $fg:A_2\rightarrow A_2$ are isomorphisms by~Proposition~\ref{prop:actionofgamma}(b).
\end{proof}

\begin{definition}\label{def:nice}
	We say $\K$ is \emph{nice} if for every separable ring $A$ of finite degree in~$\K$, there are indecomposable rings $A_1,\dots, A_n$ in~$\K$  such that $A\cong A_1 \times \ldots \times A_n$. 
\end{definition}

\begin{example}
	Let $G$ be a group and $\kk$ a field.
	The categories $\mod G$, $\Db G$ and $\stab G$ (see Section~\ref{sec:rep}) are nice categories. More generally, $\K$ is nice if it satisfies Krull-Schmidt.
\end{example}

\begin{example}
	Let $X$ be a Noetherian scheme. Then  $\D(X)$, the derived category of perfect complexes over $X$, is nice (see~Lemma~\ref{prop:noethnice}).
\end{example}

\begin{lemma}\label{lem:ringmorphism}
	Suppose $\K$ is nice and let $A$, $B$ be separable rings of finite degree in~$\K$. 
	If $B$ is indecomposable and there exists a ring morphism $A\rightarrow B$ in~$\K$, then there exists a ring morphism $C\rightarrow B$ for some indecomposable ring factor $C$ of~$A$.
\end{lemma}

\begin{proof}
	Since $\K$ is nice, we can write $A\cong A_1 \times\ldots \times A_n$ with $A_i$ indecomposable for~$1\leq i\leq n$.
	If there exists a ring morphism $A\rightarrow B$ in~$\K$, Proposition~\ref{prop:An} shows that $\1_B$ is a ring factor of $F_B(A)\cong F_B(A_1)\times \dots \times F_B(A_n)$. Since $\1_B$ is indecomposable, it is a ring factor of some $F_B(A_i)$ with $1\leq i\leq n$ by~Proposition~\ref{prop:uniquedecomposition}. 
\end{proof}

\begin{proposition}\label{prop:splittingring}
	Suppose $\K$ is nice and let $A$ be a separable ring of finite degree in~$\K$. An indecomposable ring $B$ in $\K$ is a splitting ring of $A$ if and only if $B$ is a ring factor of~$A^{[\deg(A)]}$. In particular, any separable ring in $\K$ has a splitting ring and at most finitely many.
\end{proposition}

\begin{proof}
	Let $d:=\deg(A)$ and suppose $B$ is a splitting ring of~$A$.
	By Remark~\ref{rem:BsplitsA}, $B$ is an $A^{[d]}$-algebra. Hence,
	there exists a ring morphism $C\rightarrow B$ 
	for some indecomposable ring factor $C$ of~$A^{[d]}$ by~Lemma~\ref{lem:ringmorphism}.
	Since $C$ splits~$A$, the ring morphism $C\rightarrow B$ is an isomorphism.
	Conversely, suppose $B$ is a ring factor of $A^{[d]}$, so $B$ splits~$A$.
	Let $C$ be an indecomposable separable ring splitting $A$ and suppose there is a ring morphism $C\rightarrow B$.  As before, $C$ is an $A^{[d]}$-algebra and
	there exists a ring morphism $B'\rightarrow C$ 
	for some indecomposable ring factor $B'$ of~$A^{[d]}$.
	The composition $B'\rightarrow C\rightarrow B$ is an isomorphism  
	by Remark~\ref{rem:Adsplits} and
	Lemma~\ref{lem:selfsplitfactors}.
	In other words, $B$ is a ring factor of the indecomposable ring $C$, so that~$C\cong B$.
\end{proof}

\begin{corollary}\label{cor:galoissplitring} 
	Suppose $\K$ is nice and $B$ is a separable indecomposable ring of finite degree in~$\K$. Then $B$ is quasi-Galois in $\K$ if and only if there exists a nonzero separable ring $A$ of finite degree in $\K$ such that $B$ is a splitting ring of~$A$.
\end{corollary}

\begin{proof}
	Suppose $B$ is indecomposable and quasi-Galois of degree $t$, so $B^{[2]}\cong \1_B^{\times (t-1)}$ as $B$-algebras. Then, $B$ is a splitting ring for $B$ because $B$ is a ring factor of $B^{[t]}$:
	$$B^{[t]}\cong (B^{[2]})^{[t-1]}\cong(\1_B^{\times (t-1)})^{[t-1]}\cong B^{\times (t-1)!}.$$
	Now suppose $B$ is a splitting ring for some $A$ in~$\K$, say with $\deg(A)=d>0$.
	Seeing how $F_B(B)$ is a ring factor of 
	$$F_B(A^{[d]})\cong  F_B(A)^{[d]}\cong (\1_B^{\times d})^{[d]}= \1_B^{\times d!},$$ 
	we know $F_B(B)\cong \1_B^{\times t}$ for some $t>0$.
	By~Theorem~\ref{th:galoisequivalences}, $B$ is quasi-Galois.
\end{proof}

\section{Tensor Triangular Geometry}

\begin{definition}\label{def:tt}
	A \emph{tt-category} $\K$ is an essentially small, idempotent-complete tensor-triangulated category. In particular, $\K$ comes equipped with a symmetric monoidal structure $(\T,\1)$ such that $x\T - :\K\rightarrow \K$ is exact for all objects $x$ in~$\K$. A \emph{tt-functor} $\K\rightarrow \cat L$ is an exact monoidal functor. 
\end{definition}

Throughout the rest of this paper, $(\K,\T,\1)$ will denote a tt-category.

\begin{remark}
	Balmer proved in~\cite{bseparability} that extensions along separable ring objects preserve the triangulation: $(\M A,\T_A,\1_A)$ is a tt-category, extension-of-scalars $F_A$ becomes a tt-functor and $U_A$ is exact.  
\end{remark}


\begin{definition}
	We briefly recall some tt-geometry and refer the reader to~\cite{bspectrum} for precise statements and motivation.
	The \emph{spectrum} $\Spc(\K)$ of a tt-category $\K$ is the set of all prime thick $\T$-ideals $\p\subsetneq \K$. The \emph{support} of an object $x$ in $\K$ is $\supp(x)=\{\p\in\Spc(\K)\mid x\notin \p\}\subset\Spc(\K)$. The complements $\mathcal{U}(x):=\Spc(\K)-\supp(x)$ of these supports form an open basis for the \emph{Zariski topology} on $\Spc(\K)$.
\end{definition}

\begin{remark}
	The spectrum is functorial. In particular, every tt-functor $F:\K \rightarrow \cat L$ induces a continuous map 
	$$\Spc(F):\Spc(\cat L) \longrightarrow \Spc(\K).$$
	Moreover, for all $x\in\K$, we have
	$$(\Spc F)^{-1} (\supp_{\K}(x))=\supp_{\cat L} (F(x))\subset\Spc \cat L.$$
\end{remark}

\begin{example}\label{ex:specetale}
	Let $R$ be a commutative ring. Then $\D(R)$, with left derived tensor product, is a tt-category and $\Spc(\D(R))$ recovers $\Spec(R)$.
\end{example}

Let $A$ be a separable ring in~$\K$.
We will consider the continuous map
$$f_A:=\Spc(F_A):\Spc(\M A)\longrightarrow \Spc(\K)$$
induced by the extension-of-scalars $F_A:\K\rightarrow \M A$. We collect some of its properties here.

\begin{theorem}\label{th:f_Acoequalizer}\emph{(\cite[Th.3.14]{bquillen}).}
	Let $A$ be a separable ring of finite degree in~$\K$. Then
	\begin{equation}\label{eq:f_Acoequalizer}
	\begin{tikzcd}
	\Spc(\M{(A\T A)}) \arrow[yshift=0.7ex]{r}{f_1}
	\arrow[swap, yshift=-0.7ex]{r}{f_2}&
	\Spc(\M A)
	\arrow[two heads]{rr}{f_A}
	&& \supp_{\K}(A)
	\end{tikzcd}
	\end{equation}
	is a coequaliser, where $f_1, f_2$ are the maps induced by extension-of-scalars along the morphisms $1\T \eta$ and $\eta\T 1:A\rightarrow A\T A$ respectively. In particular, the image of $f_A$ is $\supp_{\K}(A)\subset \Spc(\K)$. 
\end{theorem}

\begin{definition}\label{def:local}
	We call a tt-category $\K$ \emph{local} if $x\T y=0$ implies that $x$ or $y$ is $\T$-nilpotent for all $x,y\in\K$.
	The \emph{local category $\K_{\p} $ at the prime} $\p\in \Spc(\K)$ is the idempotent completion of the Verdier quotient $\K\diagup \p$. We write $q_{\p}$ for the canonical tt-functor $\K\twoheadrightarrow \K\diagup \p \hookrightarrow \K_{\p}$.  
\end{definition}

\begin{proposition}\label{prop:degreep}\emph{(\cite[Th.3.8]{bdegree}).}
	Suppose $A$ is a separable ring in~$\K$. If the ring $q_{\p}(A)$ has finite degree in $\K_{\p}$ for every $\p\in \Spc(\K)$, then $A$ has finite degree and 
	$$\deg_{\K}(A)=\max_{\p\in \Spc(\K)}\deg_{\K_{\p}}(q_{\p}(A)).$$
\end{proposition}

\begin{proposition}\label{prop:local}\emph{(\cite[Cor.3.12]{bdegree}).}
	Let $\K$ be a local tt-category and suppose $A, B$ are separable rings of finite degree in~$\K$. Then $\deg(A\times B)=\deg(A)+\deg(B)$.
\end{proposition}

\begin{lemma}\label{lem:degreeproperties}\emph{(\cite[Th.3.7]{bdegree}).}
	Let $A$ and $B$ be separable rings in~$\K$ and suppose $\supp(A)\subseteq \supp(B)$. Then $\deg_{\M B}(F_B(A))= \deg_{\K} (A)$.
\end{lemma}

\begin{proposition}\label{prop:noethnice}
	Suppose the spectrum $\Spc(\K)$ of $\K$ is Noetherian. Then $\K$ is nice. That is, any separable ring $A$ of finite degree in $\K$ has a decomposition $A\cong A_1 \times \ldots \times A_n$ where $A_1,\dots, A_n$ are indecomposable rings in~$\K$.
\end{proposition}

\begin{proof}
	Let $A$ be a separable ring of finite degree in~$\K$. If $A$ is not indecomposable, we can find nonzero rings $A_1, A_2 \in \K$ with $A\cong A_1\times A_2$. 
	We prove that any ring decomposition of $A$ in $\K$ has at most finitely many nonzero factors.  Suppose there is a sequence of nontrivial decompositions $A=A_1\times B_1$, $B_1=A_2\times B_2$,\ldots, with $B_n=A_{n+1}\times B_{n+1}$ for $n \geq 1$. 
	By~Proposition~\ref{prop:local}, we know $\deg (q_{\p}(B_{n}))\geq \deg (q_{\p}(B_{n+1}))$ for every $\p\in \Spc(\K)$, hence $\supp(B^{[i]}_n)\supseteq \supp(B^{[i]}_{n+1})$ for every $i\geq 0$. Seeing how $\Spc(\K)$ is Noetherian,
	we can find $k\geq 1$ with $\supp(B^{[i]}_n)=\supp(B^{[i]}_{n+1})$ for  every $i\geq 0$ and $n\geq k$.
	In particular, $\deg (q_{\p}(B_{k}))=\deg (q_{\p}(B_{k+1}))$ for every $\p\in \Spc(\K)$, so that $q_{\p}(A_{k+1})=0$ for all $\p\in \Spc(\K)$. By~Proposition~\ref{prop:degreep}, we conclude $A_{k+1}=0$.
\end{proof}

\section{Rings of Constant Degree}

\begin{definition} 
	We say a separable ring $A$ in $\K$ has \emph{constant degree} $d\in \N$ if the degree $\deg_{\K_{\p}}q_{\p}(A)$ equals $d$ for every ${\p}\in \supp(A)\subset \Spc(\K)$. 
\end{definition}

\begin{lemma}
	Let $A$ be a separable ring of degree~$d$ in~$\K$. Then $A$ has constant degree if and only if~$\supp(A^{[d]})=\supp (A)$.
\end{lemma}

\begin{proof}
	Note that $\supp(A^{[2]})\subseteq \supp(A)$ because $A\T A\cong A\times A^{[2]}$ in~$\K$. Hence $\supp(A^{[d]})\subseteq \supp(A)$.
	Now, let $\p \in\supp(A)$. Then $q_{\p}(A)$ has degree~$d$ if and only if $q_{\p}(A^{[d]})\neq 0$, in other words ${\p}\in \supp(A^{[d]})$
\end{proof}

\begin{lemma}\label{lem:lcdF}
	Let $A$ be a separable ring in $\K$ and suppose $F:\K\rightarrow \cat L$ is a tt-functor with $F(A)\neq 0$. If $A$ has constant degree~$d$, then $F(A)$ has constant degree~$d$.
	Conversely, if $F(A)$ has constant degree $d$ and $\supp(A)\subset \im(\Spc(F))$, then $A$ has constant degree~$d$.
\end{lemma}

\begin{proof}
	We first note that $\deg(F(A))\leq \deg(A)$ by~Proposition~\ref{prop:degreeproperties}(b).
	Now, if $A$ has constant degree $d$, then
	\begin{align*}
	\supp_{\cat L}(F(A)^{[d]})
	&=\supp_{\cat L}(F(A^{[d]}))
	=\Spc(F)^{-1}(\supp_{\K}(A^{[d]}))
	=\Spc(F)^{-1}(\supp_{\K}(A))\\
	&=\supp_{\cat L}(F(A))\neq \emptyset,\end{align*}
	which shows $F(A)$ has constant degree~$d$.
	Conversely, suppose $F(A)$ has constant degree $d$ and $\supp(A)\subset \im(\Spc(F))$.
	In particular, $\supp(A^{[d+1]})\subset \im(\Spc(F))$, so
	$$\emptyset=\supp(F(A^{[d+1]}))=\Spc(F)^{-1}(\supp(A^{[d+1]}))$$ implies $\supp(A^{[d+1]})=\emptyset$. Thus $A$ has degree~$d$.
	Moreover, seeing how
	\begin{align*}
	\Spc(F)^{-1}(\supp_{\K}(A^{[d]}))
	=\supp_{\cat L}(F(A)^{[d]})
	=\supp_{\cat L}(F(A))
	=\Spc(F)^{-1}(\supp_{\K}(A)),\end{align*}
	we can conclude $\supp_{\K}(A^{[d]})=\supp_{\K}(A)$.
\end{proof}

\begin{proposition}\label{prop:lcdiff}
	Let $A$ be a separable ring in~$\K$. Then $A$ has constant degree $d$ if and only if there exists a separable ring $B$ in $\K$ with $\supp(A)\subset \supp(B)$ and such that $F_B(A)\cong \1_B^{\times d}$.
	In particular, if $A$ is quasi-Galois in $\K$ with group~$\Gamma$, then $A$ has constant degree $|\Gamma|$ in~$\K$.
\end{proposition}

\begin{proof}
	If $A$ has constant degree $d$, we can let $B:=A^{[d]}$ and use~Proposition~\ref{prop:degreeproperties}(a).
	On the other hand, if $A$ and $B$ are separable rings in $\K$ with $\supp(A)\subset \supp(B)$,
	then~Theorem~\ref{th:f_Acoequalizer} and~Lemma~\ref{lem:lcdF} show that $A$ has constant degree $d$ whenever $F_B(A)$ has constant degree~$d$.
\end{proof}

\begin{proposition}\label{prop:lcddecomposition}
	Let $A$ be a separable ring of constant degree in $\K$ with connected support $\supp(A)\subset\Spc(\K)$. If $B$ and $C$ are nonzero rings in $\K$ such that $A=B\times C$, then $B$ and $C$ have constant degree and $\supp (A)=\supp (B)=\supp(C)$.
\end{proposition}

\begin{proof}
	Given that $A$ has constant degree $d$, we claim that for every $1\leq n \leq d$, $$\supp(A)=\supp(B^{[n]})\bigsqcup \supp (C^{[d-n+1]}).$$ 
	Fix $1\leq n \leq d$ and suppose ${\p}\in \supp(B^{[n]})\subset \supp(A)$, so that $\deg(q_{\p}(B))\geq n$. By~Proposition~\ref{prop:local}, $\deg(q_{\p}(C))\leq d-n$ and hence ${\p}\notin \supp (C^{[d-n+1]})$. On the other hand, if ${\p}\in \supp(A)-\supp(B^{[n]})$, we get $\deg(q_{\p}(B))\leq n-1$ and $\deg(q_{\p}(C))\geq d-n+1$. So, ${\p}\in \supp (C^{[d-n+1]})$ and the claim follows.
	
	Now, assuming $A$ has connected support,
	taking $n=\deg(B)$ shows that 
	$\supp (A)=\supp(B^{[\deg (B)]})=\supp(B)$. Similarly, letting $n=d+1-\deg(C)$ shows that $\supp (A)=\supp(C^{[\deg (C)]})=\supp(C)$. In other words, $B$ and $C$ have constant degree and $\supp (A)=\supp (B)=\supp (C)$.
\end{proof} 

\section{Quasi-Galois Theory and Tensor Triangular Geometry}

Let $A$ be a separable ring in $\K$ and $\Gamma$ a finite group of ring morphisms of~$A$. Then, $\Gamma$ acts on $\M A$ (see~Remark~\ref{rem:F_h}) and therefore on the spectrum $\Spc(\M A)$.

\begin{theorem}\label{th:galoisspectrum}
	Suppose $A$ is quasi-Galois in~$\K$ with group~$\Gamma$. Then,
	$$\supp(A)\cong \Spc(\M A)/\Gamma.$$
\end{theorem}

\begin{proof}
	Diagram~\ref {eq:lambdatriangle} yields a diagram of topological spaces
	$$
	\begin{tikzcd}[column sep=small]
	& \Spc(\M A)&\\
	\Spc(\M{(A\T A)})
	\arrow[yshift=0.5ex, xshift=-0.5ex]{ru}{f_1}
	\arrow[swap, yshift=-0.3ex, xshift=0.7ex]{ru}{f_2}
	\arrow{rr}{\cong}[swap]{l}
	&& \Spc(\prod_{\gamma\in \Gamma} \M{A_{\gamma}}),
	\arrow[swap, yshift=0.5ex, xshift=0.5ex]{lu}{g_2}
	\arrow[yshift=-0.3ex, xshift=-0.7ex]{lu}{g_1}
	\end{tikzcd}$$
	where $f_1, f_2, g_1, g_2, l$ are the maps induced by extension-of-scalars along the morphisms $1\T \eta$, $\eta\T 1$, $\varphi_1$, $\varphi_2$ and $\lambda$ respectively (in the notation of Definition~\ref{def:lambda}). 
	That is, $g_1, g_2:\bigsqcup_{\gamma\in \Gamma} \Spc(\M {A_{\gamma}})\rightarrow \Spc(\M A)$ are continuous maps such that $g_1 \inc_{\gamma}$ is the identity and $g_2 \inc_{\gamma}$ is the action of $\gamma$ on $\Spc(\M A)$.
	Now, the~coequaliser~\ref{eq:f_Acoequalizer}
	turns into
	$$\begin{tikzcd}
	\bigsqcup_{\gamma\in \Gamma} \Spc(\M {A_{\gamma}})
	\arrow[yshift=0.7ex]{r}{g_1}
	\arrow[swap, yshift=-0.7ex]{r}{g_2}&
	\Spc(\M A)
	\arrow[two heads]{rr}{f_A}
	&& \supp(A),
	\end{tikzcd}$$
	which shows $\supp(A)\cong \Spc(\M A)/\Gamma.$
\end{proof}

\begin{definition}
	Let $A$ be a ring in $\K$.
	We say $A$ is \emph{faithful} if the extension-of-scalars functor $F_A$ is faithful. We call $A$ \emph{nil-faithful} if $F_A(f)=0$ implies $f$ is $\T$-nilpotent for any morphism $f$ in $\K$.
\end{definition}

\begin{remark}
	By~\cite[Prop.3.15]{bquillen}, $A$ is nil-faithful if and only if $\supp(A)=\Spc(\K)$. 
	If $A$ is nil-faithful and quasi-Galois in~$\K$ with group~$\Gamma$, Theorem~\ref{th:galoisspectrum} recovers $\Spc(\K)$ as the $\Gamma$-orbits of $\Spc(\M A)$.
\end{remark}

\begin{remark}
	If $A$ is a faithful ring in $\K$, being quasi-Galois in~$\K$ with group~$\Gamma$ really is being \emph{Galois over} $\1$ in the sense of Auslander and Goldman~\cite[App.]{AG}.
	Indeed,
	$$\begin{tikzcd}
	\mathbb{1} \arrow[hook]{r}{\eta} & A
	\arrow[yshift=0.7ex]{r}{1\T \eta }
	\arrow[swap, yshift=-0.7ex]{r}{\eta \T 1}&
	A\T A
	\end{tikzcd}$$
	is an equaliser by~\cite[Prop.2.12]{bdescent}.
	Under the isomorphism $A\T A\cong \prod_{\gamma\in \Gamma} A_{\gamma}$, this becomes
	$$\begin{tikzcd}
	\mathbb{1} \arrow[hook]{r}{\eta} & A
	\arrow[yshift=0.7ex]{r}{\varphi_1}
	\arrow[swap, yshift=-0.7ex]{r}{\varphi_2}&
	\prod_{\gamma\in \Gamma} A_{\gamma},
	\end{tikzcd}$$
	where $\pr_{\gamma}\varphi_1=1_A$ 
	and $\pr_{\gamma}\varphi_2=\gamma$ for all $\gamma \in \Gamma$.
\end{remark}

The following lemma is a tensor-triangular version of~Lemma~\ref{lem:selfsplitfactors}.

\begin{lemma}\label{lem:ttselfsplitfactors}
	Let $A$ be a separable ring in $\K$ that splits itself. If $A_1$ and $A_2$ are indecomposable ring factors of $A$, then
	$\supp(A_1)\cap \supp(A_2)=\emptyset$ or $A_{1}\cong A_2$.
\end{lemma}

\begin{proof}
	Let $A_1$ and $A_2$ be indecomposable ring factors of $A$ and suppose $A$ splits itself. We know $F_{A_1}(A)\cong \1_{A_1}^{\times \deg(A)}$ and hence $F_{A_1}(A_2)\cong \1_{A_1}^{\times t}$ for some $t\geq 0$.
	In fact, $t=0$ only if $\supp(A_1\T A_2)=\supp(A_1)\cap\supp(A_2)=\emptyset$.
	If $t>0$, we can find a ring morphism $A_2\rightarrow A_1$ by~Proposition~\ref{prop:An}. Lemma~\ref{lem:selfsplitfactors} shows this is an isomorphism. 
\end{proof}

\begin{proposition}\label{prop:cdsplittingring}
	Suppose $\K$ is nice. Let $A$ be a separable ring in $\K$ with connected support $\supp(A)$ and constant degree. Then the splitting ring $A^{*}$ of~$A$ is unique up to isomorphism and $\supp(A)= \supp(A^{*})$.
\end{proposition}

\begin{proof}
	Let $d:=\deg(A)$. We prove that any two indecomposable ring factors, say $A_1$ and $A_2$, of $A^{[d]}$ are isomorphic.
	Note that $\supp(A)=\supp(A^{[d]})$ is connected and $A^{[d]}$ has constant degree $d!$, so that 
	$\supp(A)= \supp(A_1)= \supp(A_2)$ by~Proposition~\ref{prop:lcddecomposition}.
	Lemma~\ref{lem:ttselfsplitfactors} now shows $A_1$ and $A_2$ are isomorphic. 
\end{proof}

\begin{remark}
	In what follows, we consider a separable ring $A$ in $\K$ and assume the spectrum $\Spc(\M A)$ is connected, which implies that $A$ is indecomposable. Moreover, if the tt-category $\M A$ is \emph{strongly closed}, $\Spc(\M A)$ is connected if and only if $A$ is indecomposable, see~\cite[Th.2.11]{bfiltration}. We note that many tt-categories are strongly closed, including all examples given in this paper. 
\end{remark}

\begin{proposition}\label{prop:galoissplit}
	Suppose $\K$ is nice. Let $A$ be a separable ring in $\K$ and suppose $\Spc(\M A)$ is connected. Let $B$ be an $A$-algebra with $\supp(A)=\supp(B)$. If $B$ is quasi-Galois in~$\K$ with group $\Gamma$, then $B$ splits~$A$. 
	In particular, the degree of $A$ in $\K$ is constant.  
\end{proposition}

\begin{proof}
	If $B$ is quasi-Galois in $\K$ for some group $\Gamma$, then all of its indecomposable factors are also quasi-Galois by~Corollary~\ref{cor:galoisfactors}. What is more, $\supp(B)=f_A(\Spc(\M A))$ is connected, so the indecomposable factors of $B$ have support equal to $\supp(B)$ by~Proposition~\ref{prop:lcddecomposition}.
	It thus suffices to prove the proposition when $B$ is indecomposable.
	Now, $F_A(B)$ is quasi-Galois by~Lemma~\ref{lem:qgextensions} and 
	$\supp(F_A(B))=f_A^{-1}(\supp(B))=\Spc(\M A)$ is connected.
	By~Corollary~\ref{cor:BringfactorF_A(B)},	
	$\overline{B}$ is an indecomposable ring factor of $F_A(B)$, and
	all ring factors of $F_A(B)$ have equal support by~Proposition~\ref{prop:lcddecomposition}.
	Lemma~\ref{lem:ttselfsplitfactors} now shows that
	$F_A(B)\cong \overline{B}^{\times t}$ for some $t\geq 1$.
	Forgetting the $A$-action, we get
	$A\T B\cong B^{\times t}$ in $\K$ and
	$F_B(A\T B)\cong F_B(B^{\times t})\cong \1_B^{\times dt}$ in $\M B$, where $d:=\deg(B)$.
	On the other hand,
	$F_B(A\T B)\cong F_B(A)\T_B \1_B^{\times d} \cong (F_B(A))^{\times d}$.
	It follows that $F_B(A)\cong \1_B^{\times t}$, with $t=\deg(A)$ by~Lemma~\ref{lem:degreeproperties}.
\end{proof}

\begin{theorem}\label{th:splittingring}\emph{(Quasi-Galois Closure).}
	Suppose $\K$ is nice. Let $A$ be a separable ring of constant degree in $\K$ and suppose $\Spc(\M A)$ is connected. The splitting ring $A^{*}$ is the quasi-Galois closure of~$A$. That is,
	$A^{*}$ is quasi-Galois in~$\K$, $\supp(A)=\supp(A^{*})$ and for any  $A$-algebra $B$ that is quasi-Galois in~$\K$ with $\supp(A)=\supp(B)$, there exists a ring morphism $A^{*}\rightarrow $B.
\end{theorem}

\begin{proof} 
	Corollary~\ref{cor:galoissplitring} and~Proposition~\ref{prop:cdsplittingring} show that $A^{*}$ is quasi-Galois in~$\K$ and $\supp(A)=\supp(A^{*})$.
	Suppose there is an $A$-algebra $B$ as above. By~Proposition~\ref{prop:galoissplit}, $B$ splits~$A$, so there exists a ring morphism $A^{[\deg(A)]} \rightarrow B$. The result now follows because $A^{[\deg(A)]}\cong A^{*}\times\ldots \times A^{*}$ by~Proposition~\ref{prop:cdsplittingring}.
\end{proof}

\begin{remark}
	By~Proposition~\ref{prop:galoissplit}, the assumption that $A$ has constant degree is necessary for the existence of a quasi-Galois closure $A^{*}$ of $A$ with $\supp(A)=\supp(A^{*})$. 
\end{remark}

\section{Some Modular Representation Theory}\label{sec:rep}

Let $G$ be a finite group and $\kk$ a field with characteristic $p>0$ dividing~$|G|$. We write $\mod G$ for the category of finitely generated left $\kk G$-modules. This category is nice, idempotent-complete and symmetric monoidal: the tensor is $\T_{\kk}$ with diagonal G-action, and the unit is the trivial representation~$\1 = \kk$. 

We will also work in the bounded derived category $\Db G$ and stable category $\stab G$, which are nice tt-categories.
The spectrum $\Spc(\Db G)$ of the derived category is homeomorphic to the homogeneous spectrum $\Spec^{h}(H^{\bullet}(G,\kk))$ of the  graded-commutative cohomology ring $H^{\bullet}(G,\kk)$.
Accordingly, the spectrum $\Spc(\stab G)$ of the stable category is homeomorphic to the projective support variety $\V_G(\kk):=\Proj(H^{\bullet}(G,\kk))$, see~\cite{bcr}.

\begin{notation}
	Let $H\leq G$ be a subgroup.
	The $\kk G$-module $A_H=A^G_{H}:=\kk(G/H)$ is the free $\kk$-module with basis $G/H$ and left $G$-action given by~$g\cdot [x]=[gx]$ for every~$[x]\in G/H$. 
	The $\kk G$-linear map
	$\mu: A_{H} \T_{\kk} A_{H}\longrightarrow A_{H}$
	is given by
	$$\begin{tikzcd}
	\gamma\T \gamma'\arrow[mapsto]{r} 
	& \left\lbrace \begin{matrix} \gamma & \text{if } \gamma= \gamma'
	\\0 & \text{if } \gamma\neq \gamma'\end{matrix}\right.  
	\hspace{1cm} \text{for all $\gamma, \gamma' \in G/H$.} \end{tikzcd}$$
	We define $\eta:\1 \rightarrow A_{H}$ by sending 
	$1\in \kk$ to $\sum_{\gamma \in G/H}\gamma\in \kk(G/H)$.

	We will write $\K (G)$ to denote $\mod G$, $\Db G$ or $\stab G$ and consider the object $A_{H}$ in each of these categories. 
\end{notation} 

\begin{proposition}\label{prop:AH}\cite[Prop.3.16, Th.4.4]{bstacks}
	Let $H \leq G$  be a subgroup. Then, 
	\begin{enumerate}[label=(\alph*)]
		\item The triple $(A_{H},\mu, \eta)$ is a commutative separable ring object in $\K(G)$.

		\item  There is an equivalence of categories 
		$$ \Psi^G_H : \K(H) \xrightarrow{\simeq} \Mod{\K(G)}{A_{H}}$$
		sending $V\in \K(H)$ to 
		$\kk G \T_{\kk H}V \in \K(G)$ with $A_H$-action
		$$\varrho:\kk (G/H)\T_{\kk} (\kk G \T_{\kk H}V)
		\longrightarrow
		\kk G \T_{\kk H}V $$
		given for $\gamma\in G/H$, $g\in G$ and $v\in V$ by
		$\gamma\T g\T v 
		\longmapsto 
		 \left\lbrace \begin{matrix} g\T v & \text{if $g\in \gamma$} 
		\\0 & \text{if $g\notin \gamma$}\end{matrix}\right.$.

		\item The following diagram commutes up to isomorphism:
		$$\begin{tikzcd}[column sep=huge]
		& \K(G)
		\arrow{rd}{F_{A_{H}}}
		\arrow[swap]{ld}{\Res {G} {H}}\\
		\K(H) 
		\arrow{rr}{\Psi^G_H}[swap]{\simeq}
		& & \Mod{\K(G)}{A_{H}}.
		\end{tikzcd}$$
	\end{enumerate}
\end{proposition} 

So, every subgroup $H\leq G$ provides an indecomposable separable ring $A_{H}$ in~$\K (G)$, along which extension-of-scalars becomes restriction to the subgroup. 

\begin{proposition}\label{prop:degreeDb}
	The ring $A_{H}$ has degree $[G:H]$ in $\mod G$ and $\Db G$.
\end{proposition}

\begin{proof}
	Seeing how the fiber functor $\Res G {\{1\}}$ is conservative, we get
	$$\deg_{\mod G}(A_H)=\deg_{\mod {}}(\Res G {\{1\}}(A_H))=[G:H].$$
	The degree of $A_H$ in $\Db G$ is computed in ~\cite[Cor.4.5]{bdegree}.
\end{proof}

\begin{lemma}\label{lem:suppAH} 
	Let $\K(G)$ denote $\Db G$ or $\stab G$ and consider subgroups $K\leq H\leq G$. Then $\supp (A_H)=\supp(A_K)\subset\Spc(\K(G))$ if and only if every elementary abelian $p$-subgroup of $H$ is conjugate in $G$ to a subgroup of~$K$.
\end{lemma}

\begin{proof}
	This follows from~\cite[Th.9.1.3 ]{evens}, seeing how
	$\supp(A_H)=(\Res G H)^{*}(\Spc(\K(H)))$ can be written as a union of disjoint pieces coming from conjugacy classes in $G$ of elementary abelian $p$-subgroups of~$H$.
\end{proof}

\begin{notation}
	For any two subgroups $H,K\leq G$, we write $_{H}[g]_{K}$ for the equivalence class of $g\in G$ in $H\backslash G /K$, just $[g]$ if the context is clear. We will write $H^{g}:=g^{-1}H g$ for the conjugate subgroups of~$H$.
\end{notation}

\begin{remark}\label{rem:mackey}
	Let $H,K\leq G$ be subgroups and choose a complete set $T\subset G$ of representatives for $H\backslash G/K$. Consider the Mackey isomorphism of $G$-sets,
	$$\coprod_{g\in T} G/(K \cap H^g)\xrightarrow{\cong} G/K \times G/H,$$
	sending $[x]\in G/(K \cap H^g)$ to $([x]_{K},[xg^{-1}]_H)$.
	The corresponding ring isomorphism 
	$$\tau: A_{K} \T A_{H}  \xrightarrow{\cong}\prod_{g\in T} A_{K \cap H^g},$$
	in~$\K(G)$, is given for $g\in T$ and $x,y\in G$ by 
	$$ \pr_g \tau \: ([x]_K\T [y]_H)
	= \left\lbrace \begin{matrix} [xk]_{K\cap H^g} 
	& \text{if } _{H}[g]_{K}={}_{H}[y^{-1}x]_{K} 
	\\0 & \text{otherwise }, \end{matrix}\right.  $$
	with $k\in K $ such that $y^{-1}x kg^{-1} \in H$.
	This yields an $A_K$-algebra structure on $A_{K \cap H^t}$ for every $t\in T$, given by
	$$A_K\xrightarrow{1\T\eta} A_{K} \T A_{H}  \cong\prod_{g\in T} A_{K \cap H^g}\xrightarrow{\pr_{t}} A_{K \cap H^t},$$
	which sends $[x]_K\in G/K$ to $\sum\limits_{\substack{[k]\in K/K\cap H^t}}[xk]_{K\cap H^t}\in A_{K \cap H^t}$.
	In the notation of~Proposition~\ref{prop:AH}(b), this means
	$F_{A_K}(A_H)\cong \prod_{g\in T}\Psi^G_K(A^K_{K\cap H^g})$ in $\Mod{\K(G)}{A_{K}}$.
\end{remark}

\begin{lemma}\label{lem:cosets}
	Let $H<G$. Suppose $x,g_1, g_2, \ldots, g_n\in G$ and $1\leq i\leq n$. Then
	$${}_{H}[x]_{H\cap H^{g_1}\cap\ldots\cap H^{g_n}}={}_{H}[g_i]_{H\cap H^{g_1}\cap\ldots\cap H^{g_n}}$$ if and only if $_{H}[x]=  {}_{H}[g_i]$.
\end{lemma}

\begin{proof}
	It suffices to prove that for $x,y\in G$,
	we have $_{H}[x]_{H^y}= {}_{H}[y]_{H^y}$ if and only if $_{H}[x]=  {}_{H}[y]$.
	This follows because for $[x]=[y]$ in $H\backslash G / H^y$, there are $h,h'\in H$ with $x=hy(y^{-1}h'y)=hh'y$. 
\end{proof}

\begin{notation}
	We fix a subgroup $H<G$ and a complete set $S\subset G$ of representatives for $H\backslash G/H$.
	Likewise, if $g_1, g_2, \ldots, g_n \in G$ we will write $S_{g_1, g_2, \ldots, g_n}\subset G$ to denote some complete set of representatives for
	$H\backslash G /H\cap H^{g_1}\cap\ldots\cap H^{g_n}$.
	
	Recall that $\K (G)$ can denote $\mod G$, $\Db G$ or $\stab G$.
\end{notation}

\begin{lemma}\label{lem:AHn}
	Let $1\leq n < [G:H]$. There is an isomorphism of rings 
	$$A_H^{[n+1]}\cong 
	\prod_{g_1,\ldots,g_n} 
	A_{{H\cap H^{g_1}\cap\ldots\cap H^{g_{n}}}}$$
	in $\K(G)$, where the product runs over all $g_1\in S$ and $g_{i}\in S_{g_{1}, \ldots, g_{i-1}}$ for $2\leq i\leq n$ with ${}_H \! [1],{}_H \! [g_1],\ldots, {}_H \! [g_{n}]$ distinct in $H\backslash G$.
\end{lemma}

\begin{proof}
	By Remark~\ref{rem:mackey}, we know that
	$$ A_{H} \T A_{H}  \cong
	\prod_{g\in S} A_{H\cap H^g}
	= A_{H}\times 
	\prod_{\substack{g\in S\\ {}_H \! [g]\neq {}_H \! [1]}} A_{H\cap H^g}$$
	so~Proposition~\ref{prop:uniquedecomposition} shows
	$A_H^{[2]}\cong 
	\prod\limits_{\substack{g\in S\\ {}_H \! [g]\neq {}_H \! [1]}} A_{H\cap H^g}$ in $\K(G)$.
	Now suppose 
	$$A_H^{[n]}\cong 
	\prod_{g_1,\ldots,g_{n-1}}
	A_{{H\cap H^{g_{1}}\cap\ldots\cap H^{g_{n-1}}}}$$
	for some $1\leq n<[G:H]$,
	where the product runs over all $g_1\in S$ and $g_{i}\in S_{g_{1}, \ldots, g_{i-1}}$ for $2\leq i\leq n-1$ with ${}_H \! [1],{}_H \! [g_1],\ldots, {}_H \! [g_{n-1}]$ distinct in $H\backslash G$.
	Then 
	\begin{eqnarray*}
	A_{H}^{[n]} \T A_{H} 
	 \cong
	\prod_{g_1,\ldots,g_{n-1}}
	A_{{H\cap H^{g_1}\cap\ldots\cap H^{g_{n-1}}}}\T A_H
	\cong
	\prod_{g_1,\ldots,g_{n-1}}\;\;
	\prod_{g_{n}\in S_{g_1,\ldots,g_{n-1}}}
	A_{H\cap H^{g_1}\cap\ldots\cap H^{g_{n}}},
	\end{eqnarray*}
	again by Remark~\ref{rem:mackey}.
	We note that every  $g_n\in S_{g_1,\ldots,g_{n-1}}$ with $_H[g_n]={}_H[1]$ or 	$_H[g_n]={}_H[g_i]$ for $1\leq i\leq n-1$ provides a copy of $A_{H}^{[n]}$.
	By~Lemma~\ref{lem:cosets}, this happens exactly $n$ times.
	Hence,
	$$A_{H}^{[n]} \T A_{H} 
	\cong
	\left( A_{H}^{[n]}\right)^{\times n} \times
	\prod_{g_1,\ldots,g_n} 
	A_{H\cap H^{g_1}\cap\ldots\cap H^{g_{n}}},$$
	where the product runs over all $g_1\in S$ and $g_{i}\in S_{g_{1}, \ldots, g_{i-1}}$ for $2\leq i\leq n$ with ${}_H \! [1],{}_H \! [g_1],\ldots, {}_H \! [g_{n}]$ distinct in $H\backslash G$.
	The lemma follows by~Proposition~\ref{prop:degreeproperties}(a).
\end{proof}

\begin{corollary}\label{cor:AHd}
	Let $d:=[G:H]$.
	There is an isomorphism of rings 
	$$A_H^{[d]}\cong
	\left( A_{\norm H}\right)^{\times \frac{d!}{[G:\norm H]}}$$
	in $\mod G$ and $\Db G$. 
	Here, $\norm H:=\bigcap\limits_{\substack{g\in G}}g^{-1}Hg$ is the normal core of~$H$ in~$G$. 
\end{corollary}

\begin{proof}
	From the above lemma, we know 
	$A_H^{[d]}\cong
	\prod\limits_{\substack{g_1,\ldots,g_{d-1}}} 
	A_{{H\cap H^{g_1}\cap\ldots\cap H^{g_{d-1}}}}$,
	where the product runs over some $g_1,\dots g_{d-1}\in G$ with $\left\lbrace {}_H \! [1],{}_H \! [g_1],\ldots, {}_H \! [g_{d-1}]\right\rbrace =~H\backslash G$.
	This shows 
	$A_H^{[d]}\cong A_{\norm H}^{\times t}$
	for some $t\geq 1$.
	Now, $\deg(A_{\norm H})=[G:\norm H]$ and $\deg(A_H^{[d]}) =d!$ by Remark~\ref{rem:Adsplits}, so
	$t=\frac{d!}{[G:\norm H]}$
	by~Lemma~\ref{cor:1factor}(c).
\end{proof}

\begin{corollary}\label{cor:Dbqg}
	The ring $A_H$ in $\Db G$ has constant degree $[G:H]$ if and only if $\norm H$ contains every elementary abelian $p$-subgroup of~$H$. 
	In that case, 
	its quasi-Galois closure is $A_{\norm H}$. 
	Furthermore, $A_H$ is quasi-Galois in $\Db G$ if and only if $H$ is normal in~$G$.  
\end{corollary}

\begin{proof}
	The first statement follows immediately from~Lemma~\ref{lem:suppAH} and~Corollary~\ref{cor:AHd}. By~Proposition~\ref{prop:splittingring}, the splitting ring of $A_H$ is $A_{\norm H}$, so the second statement is~Theorem~\ref{th:splittingring}.
	Since $A_H$ is an indecomposable ring, it is quasi-Galois if and only if it is its own splitting ring. 
	Hence $A_H$ is quasi-Galois if and only if $A_{\norm H}\cong A_H$, which yields $\norm H= H$ by comparing degrees. 
\end{proof}

\begin{remark}\label{rem:kgstab}
	Let $H\leq G$ be a subgroup. Recall that $A_H\cong 0$ in $\stab G$ if and only if $p$ does not divide $|H|$. On the other hand, $A_H\cong \kk$ in $\stab G$ if and only if $H$ is strongly $p$-embedded in $G$, that is $p$ divides $|H|$ and $p$ does not divide $|H\cap H^g|$ if $g\in G- H$.
\end{remark}

\begin{proposition}\label{prop:stabqg}
	Let $H\leq G$ and consider the ring $A_H$ in $\stab G$. Then,
	\begin{enumerate}[label=(\alph*)]
		\item The degree of $A_H$ is the greatest $0\leq n\leq [G:H]$ such that there exist distinct $[g_1], \ldots, [g_n]$ in $H\backslash G$ with $p$ dividing $|H^{g_1}\cap\ldots\cap H^{g_n}|$. 

		\item The ring $A_H$ is quasi-Galois if and only if $p$ divides $|H|$ and $p$ does not divide $\lvert H\cap H^g \cap H^{gh}\rvert$ whenever $g\in G- H$ and $h\in H- H^g$.

		\item If $A_H$ has degree $n$, the degree is constant if and only if there exist distinct $[g_1], \ldots, [g_n]$ in $H\backslash G$ such that $H^{g_1}\cap\ldots\cap H^{g_{n}}$ contains a $G$-conjugate of every elementary abelian $p$-subgroup of~$H$.\\
		In that case, $A_H$ has quasi-Galois closure given by $A_{H^{g_1}\cap\ldots\cap H^{g_{n}}}$.
	\end{enumerate}
\end{proposition}

\begin{proof}
	For (a), recall that $\deg(A_H)$ is the greatest $n$ such that $A_H^{[n]}\neq 0$, thus such that
	there exist distinct
	${}_H \! [1],{}_H \! [g_1],\ldots, {}_H \! [g_{n-1}]$ 
	with $|H\cap H^{g_1}\cap\ldots\cap H^{g_{n-1}}|$ divisible by~$p$. 
	To show (b), recall that $F_{A_H} (A_{H})  \cong
	\prod_{g\in S} \Psi^G_H (A^H_{H\cap H^g})$ by~Remark~\ref{rem:mackey}.
	It follows that $ F_{A_H} (A_{H})  \cong
	\1_{A_H}^{\times \deg (A_H)}$ in $\Mod{\stab G}{A_H}$ if and only if $$\prod_{g\in S} A^H_{H\cap H^g}
	\cong \kk^{\times \deg (A_H)}$$ in $\stab H$.
	So, $A_H$ is quasi-Galois in $\stab G$ if and only if $A_H\neq 0$ and for every $g\in G$, either $A^H_{H\cap H^g}=0$ or $A^H_{H\cap H^g}\cong \kk$ in $\stab H$. By~Remark~\ref{rem:kgstab}, this means 
	either $p\ndivides |H\cap H^g|$, or $p\divides|H\cap H^g|$ but $p\ndivides|H\cap H^g\cap H^{gh}|$ when $h\in H- H^g$.
	Equivalently, $p$ does not divide
	$|H\cap H^g\cap H^{gh}|$ whenever $g\in G-H$ and $h\in H- H^g$.
	For (c), suppose $A_H$ has constant degree $n$.
	By~Proposition~\ref{prop:cdsplittingring}, any indecomposable ring factor of $A_H^{[n]}$ is isomorphic to 
	the splitting ring $A_H^{*}$, so~Lemma~\ref{lem:AHn} shows that the quasi-Galois closure is given by $A_H^{*}\cong
	A_{H^{g_1}\cap\ldots\cap H^{g_{n}}}$ for any distinct ${}_H \! [g_1],\ldots, {}_H \! [g_{n}]$  with $|H^{g_1}\cap\ldots\cap H^{g_{n}}|$ divisible by~$p$. 
	Then, $\supp(A_H)=\supp(A_H^{*})=\supp(A_{H^{g_1}\cap\ldots\cap H^{g_{n}}})$ so $H^{g_1}\cap\ldots\cap H^{g_{n}}$ contains a $G$-conjugate of every elementary abelian $p$-subgroup of~$H$.  
	On the other hand, if there exist distinct $[g_1], \ldots, [g_n]$ in $H\backslash G$ such that $H^{g_1}\cap\ldots\cap H^{g_{n}}$ contains a $G$-conjugate of every elementary abelian $p$-subgroup of~$H$, then  $\supp(A_H^{[n]})=\supp(A_{H^{g_1}\cap\ldots\cap H^{g_{n}}})=\supp(A_H)$, so the degree of $A_H$ is constant.
\end{proof}


\begin{example}
	Let $p=2$ and suppose $G=S_4$ is the symmetric group on~$\{1,2,3,4\}$. If $H\cong S_3$ is the subgroup of permutations fixing $\{4\}$, the ring $A_H$ in $\stab G$ has constant degree $2$.
	Indeed, the intersections $H\cap H^g$ with $g\in G-H$ each fix two elements of $\{1,2,3,4\}$ pointwise, and the intersections $H\cap H^{g_1}\cap H^{g_2}$ with $[1], [g_1], [g_2]$ distinct in $H\backslash G$ are trivial. The quasi-Galois closure of $A_H$ in $\stab G$ is $A_{S_2}$, with $S_2$ embedded in $H$.
\end{example}

\end{document}